\documentclass[10pt]{amsart}
\usepackage{amssymb, amsthm, amsmath}
\usepackage[T1]{fontenc}
\usepackage[all]{xy}
\usepackage{graphicx}
\usepackage{psfrag}
\usepackage{latexsym}

\newtheorem{prop}{Proposition}
\newtheorem{thm}{Theorem}
\newtheorem{cor}{Corollary}
\newtheorem{lemma}{Lemma}
\theoremstyle{definition}

\newtheorem{defn}{Definition}

\newtheorem{remark}{Remark}

\newcommand\C{{\mathbb C}}

\newcommand\N{{\mathbb N}}

\newcommand\IH{{\mathbb H}}
\newcommand{\ti}{\vartheta}

\newcommand{\Ti}{\Theta}
\newcommand{\Eta}{H}

\newcommand\cC{{\mathcal C}}

\newcommand\tm{{\mathrm{t}}}
\newcommand\Z{{\mathbb Z}}
\newcommand\cP{{\mathcal P}}

\newcommand\AS{{\mathfrak S}}

\newcommand\supp{{\mathrm{supp}}}

\newcommand\al{\alpha}
\newcommand{\be}{\beta}
\newcommand\la{\lambda}

\newcommand\noin{\noindent}

\newcommand\bull{{\scriptscriptstyle \bullet}}
\newcommand\eqto{\stackrel{\lower1.5pt\hbox{$\scriptstyle\sim\,$}}\to}
\newcommand\ov{\overline}

\newcommand\wh{\widehat}
\newcommand\wt{\widetilde}

\DeclareMathOperator{\SO}{SO}

\DeclareMathOperator{\OG}{OG}

\DeclareMathOperator{\HH}{\mathrm{H}}

\DeclareMathOperator{\type}{\mathrm{type}}

\newcommand{\ignore}[1]{}

\begin{document}

\title[Double eta polynomials and equivariant Giambelli formulas]
{Double eta polynomials and equivariant Giambelli formulas}

\date{May 24, 2016}

\author{Harry~Tamvakis} \address{University of Maryland, Department of
Mathematics, 1301 Mathematics Building, College Park, MD 20742, USA}
\email{harryt@math.umd.edu}

\subjclass[2010]{Primary 14N15; Secondary 05E15, 14M15}

\thanks{The author was supported in part by NSF Grant DMS-1303352.}

\begin{abstract}
We use Young's raising operators to introduce and study {\em double
eta polynomials}, which are an even orthogonal analogue of Wilson's
double theta polynomials. Our double eta polynomials give Giambelli
formulas which represent the equivariant Schubert classes in the
torus-equivariant cohomology ring of even orthogonal Grassmannians,
and specialize to the single eta polynomials of Buch, Kresch, and
the author.
\end{abstract}

\maketitle

\section{Introduction}
\label{intro}

Let $k$ be a positive integer and $\OG=\OG(n-k,2n)$ be the
Grassmannian which parametrizes isotropic subspaces of dimension $n-k$
in the vector space $\C^{2n}$, equipped with an orthogonal form.  The
eta polynomials $\Eta_\la(c)$ of Buch, Kresch, and the author
\cite{BKT2, T2} are Giambelli polynomials which represent the Schubert
classes in the cohomology ring of $\OG$. Our aim here is to define
{\em double eta polynomials} $\Eta_\la(c\, |\, t)$ which represent the
equivariant Schubert classes in the equivariant cohomology ring
$\HH^*_T(\OG)$, where $T$ is a maximal torus of the complex even
orthogonal group. The companion theory of double theta polynomials for
the symplectic and odd orthogonal Grassmannians was provided in
\cite{TW}; we refer the reader there for more information, and 
to \cite{T1, T2} for the solution of the equivariant Giambelli 
problem in general, for any isotropic partial flag variety.

The Schubert classes on $\OG(n-k,2n)$ are parametrized by the
$k$-Grassmannian elements of the Weyl group $\wt{W}_n$ for the root
system $\text{D}_n$. The group $\wt{W}_n$ is the subgroup of
the hyperoctahedral group consisting of all signed permutations with
an even number of sign changes. We define the embedding
$\wt{W}_n\hookrightarrow \wt{W}_{n+1}$ by adjoining the fixed point
$n+1$, let $\wt{W}_\infty := \cup_n \wt{W}_n$, and work initially in
the latter group. An element $w=(w_1,w_2,\ldots)$ of $\wt{W}_\infty$
is {\em $k$-Grassmannian} if and only if
\[
|w_1| < w_2 <\cdots < w_k \quad \text{and} \quad w_{k+1}<w_{k+2}<\cdots .
\]

Our Giambelli formulas require the equivalent parametrization of the
Schubert classes by the typed $k$-strict partitions of \cite{BKT1}. An
integer partition $\la=(\la_1,\ldots,\la_\ell)$ is {\em $k$-strict} if
no part $\la_j$ greater than $k$ is repeated.  A {\em typed $k$-strict
  partition} is a pair consisting of a $k$-strict partition $\la$
together with an integer $\type(\la)\in \{0,1,2\}$, which is positive
if and only if $\la_j=k$ for some index $j$.

There is a bijection between the $k$-Grassmannian elements of
$\wt{W}_\infty$ and typed $k$-strict partitions, obtained as follows.
If the element $w$ corresponds to the typed partition $\la$, then for
each $j\geq 1$,
\begin{equation}
\label{laweq}
\la_j=\begin{cases} 
k-1+|w_{k+j}| & \text{if $w_{k+j}<0$}, \\
\#\{p\leq k\, :\, |w_p|> w_{k+j}\} & \text{if $w_{k+j}>0$}
\end{cases}
\end{equation}
while $\type(\la)>0$ if and only if $|w_1|>1$, and
in this case $\type(\la)$ is equal to $1$ or $2$ depending on whether
$w_1>0$ or $w_1<0$, respectively. Using this bijection, we attach to 
any typed $k$-strict partition $\la$ a finite set of 
pairs\,\footnote{The condition $w_{k+i}+ w_{k+j} < 0$ in (\ref{Cweq}) is 
equivalent to $\la_i+\la_j \geq 2k+j-i$.}
\begin{equation}
\label{Cweq}
\cC(\la) := \{ (i,j)\in \N\times\N \ |\ 1\leq i<j \ \ \text{and} \ \ 
 w_{k+i}+ w_{k+j} < 0 \}
\end{equation}
and a sequence $\beta(\la)=\{\beta_j(\la)\}_{j\geq 1}$ defined by
\begin{equation}
\label{css2w}
\beta_j(\la):=\begin{cases}
w_{k+j}+1 & \text{if $w_{k+j}<0$}, \\
w_{k+j} & \text{if $w_{k+j}>0$}.
\end{cases}
\end{equation}
For example, the $3$-Grassmannian element 
$w = (-4, 6, 8, -5, -2, -1, 3, 7)$ of $\wt{W}_8$ 
corresponds to the 
$3$-strict partition $\la=(7,4,3,3,1)$ of type 2, and we have 
$\cC(\la)=\{(1,2), (1,3), (1,4), (2,3)\}$ and 
$\beta(\la) = (-4,-1,0,3,7)$.

Let $t=(t_1,t_2,\ldots)$ be a list of commuting variables and $z$ be a
formal variable. For any integers $j\geq 0$ and $r\geq 1$, the
elementary and complete symmetric polynomials $e_j(t_1,\ldots,t_r)$
and $h_j(t_1,\ldots,t_r)$ are defined by the generating series
\[
\prod_{i=1}^r(1+t_iz) = \sum_{j=0}^{\infty}
e_j(t_1,\ldots,t_r)z^j \ \ \ \text{and} \ \ \ 
\prod_{i=1}^r(1-t_iz)^{-1} = \sum_{j=0}^{\infty}
h_j(t_1,\ldots,t_r)z^j,
\]
respectively. Let $e^r_j(t):=e_j(t_1,\ldots,t_r)$,
$h^r_j(t):=h_j(t_1,\ldots,t_r)$, and $e^0_j(t)=h^0_j(t)=\delta_{0,j}$,
where $\delta_{0,j}$ denotes the Kronecker delta. Furthermore, if
$r<0$ then define $h^r_j(t):=e^{-r}_j(t)$.  Let
$b=(\wt{b}_k,b_1,b_2,\ldots)$ and $c=(c_1,c_2,\ldots)$ be two
further families of commuting variables, and set $c_0=b_0=1$ and
$c_p=b_p=0$ for any $p<0$. These variables are related by the
equations
\begin{equation}
\label{ctob}
c_p=
\begin{cases}
b_p &\text{if $p< k$},\\
b_k+\wt{b}_k &\text{if $p=k$},\\
2b_p &\text{if $p> k$}.
\end{cases}
\end{equation}
For any $p,r\in \Z$ and for $s\in \{0,1\}$, 
define the polynomials $c^r_p$ and $a^s_p$ by
\[
c^r_p:= \sum_{j=0}^p c_{p-j} \, h^r_j(-t) \ \ \ \text{and} \ \ \
a^s_p:= \frac{1}{2}c_p+\sum_{j=1}^p c_{p-j} \, h^s_j(-t).
\]
Moreover, define
\[
b^s_k:= b_k+\sum_{j=1}^k c_{k-j} \, h^s_j(-t) \ \ \ \text{and} \ \ \
\wt{b}^s_k:= \wt{b}_k+\sum_{j=1}^k c_{k-j} \, h^s_j(-t).
\]

An integer sequence $\al = (\al_1,\al_2,\ldots)$ is assumed to have
finite support when it appears as a subscript.  For any integer
sequences $\al$ and $\rho$, let 
\[
\wh{c}_\al^\rho:= \wh{c}_{\al_1}^{\rho_1}\wh{c}_{\al_2}^{\rho_2}\cdots
\]
where, for each $i\geq 1$, 
\[
\wh{c}_{\al_i}^{\rho_i}:= c_{\al_i}^{\rho_i} + 
\begin{cases}
(2\wt{b}_k-c_k)e^{\al_i-k}_{\al_i-k}(-t) & \text{if $\rho_i = k - \al_i < 0$ 
and $i$ is odd}, \\
(2b_k-c_k)e^{\al_i-k}_{\al_i-k}(-t) & \text{if $\rho_i = k - \al_i < 0$ 
and $i$ is even}, \\
0 & \text{otherwise}.
\end{cases}
\]

The eta polynomials are defined using Young's raising operators
\cite{Y}. The basic operator $R_{ij}$ for $i<j$ acts on an integer sequence 
$\alpha$ by the prescription 
$$R_{ij}(\alpha) := (\alpha_1,\ldots,\alpha_i+1,\ldots,\alpha_j-1,
\ldots).$$ A {\em raising operator} $R$ is any finite monomial in the
basic operators $R_{ij}$. If $R:=\prod_{i<j} R_{ij}^{n_{ij}}$ is any
raising operator and $m \geq 1$, denote by $\supp_m(R)$ the set of all
indices $i$ and $j$ such that $n_{ij}>0$ and $j<m$.  For any typed
$k$-strict partition $\la$, we consider the raising operator
expression $R^\la$ given by
\begin{equation}
\label{Req}
R^{\la} := \prod_{i<j}(1-R_{ij})\prod_{(i,j)\in\cC(\la)}
(1+R_{ij})^{-1}.
\end{equation}

\begin{defn}
\label{Etadef} 
Let $\la$ be a typed $k$-strict partition of length $\ell$, let
$\ell_k(\la)$ denote the number of parts $\la_i$ which are strictly
greater than $k$, let $m:=\ell_k(\la)+1$ and $\be:=\be(\la)$.  Let $R$
be any raising operator appearing in the expansion of the power series
$R^\la$ and set $\nu:=R\la$. If $\type(\la)=0$, then define
\[
R \star \wh{c}^\be_{\la} = \ov{c}^\be_{\nu} :=
\ov{c}_{\nu_1}^{\be_1}\cdots\ov{c}^{\be_\ell}_{\nu_\ell}
\]
where, for each $i\geq 1$, 
\[
\ov{c}_{\nu_i}^{\be_i}:= 
\begin{cases}
c_{\nu_i}^{\be_i} & \text{if $i\in\supp_m(R)$}, \\
\wh{c}_{\nu_i}^{\be_i} & \text{otherwise}.
\end{cases}
\]
If $\type(\la)>0$ and $R$ involves any factors $R_{ij}$ with $i=m$ or
$j=m$, then define
\[
R \star \wh{c}^\be_{\la} := \ov{c}_{\nu_1}^{\be_1} \cdots 
\ov{c}_{\nu_{m-1}}^{\be_{m-1}} \, a^{\be_m}_{\nu_m} \, c_{\nu_{m+1}}^{\be_{m+1}}
\cdots c^{\be_\ell}_{\nu_\ell}.
\]
If $R$ has no such factors, then define
\[
R \star \wh{c}^\be_{\la} := \begin{cases}
\ov{c}_{\nu_1}^{\be_1} \cdots 
\ov{c}_{\nu_{m-1}}^{\be_{m-1}} \, b^{\be_m}_k \, c_{\nu_{m+1}}^{\be_{m+1}}
\cdots c^{\be_\ell}_{\nu_\ell} & 
\text{if  $\,\type(\la) = 1$}, \\
\ov{c}_{\nu_1}^{\be_1} \cdots
\ov{c}_{\nu_{m-1}}^{\be_{m-1}} \, \wt{b}^{\be_m}_k \, c_{\nu_{m+1}}^{\be_{m+1}} 
\cdots c^{\be_\ell}_{\nu_\ell} 
& \text{if  $\,\type(\la) = 2$}.
\end{cases}
\]
Define the {\em double eta polynomial} $\Eta_\la(c \, |\, t)$ by 
\[
\Eta_\la(c \, |\, t) := 2^{-\ell_k(\la)}R^\la\star\wh{c}^{\be(\la)}_{\la}.
\]
The single eta polynomial $\Eta_\la(c)$ of \cite{BKT2} is given by
$\Eta_\la(c)=\Eta_\la(c \, |\, 0)$. 
\end{defn}

Table \ref{table1} lists the double eta polynomials indexed by the
$1$-Grassmannian and $2$-Grassmannian elements in $\wt{W}_3$.  In the
table, the symbols $e^r_j$ and $h^r_j$ are used to denote $e^r_j(-t)$
and $h^r_j(-t)$, respectively. As is customary, a bar over an
integer is used to denote a negative sign.

{\small{
\begin{table}[t]
\caption{Double eta polynomials for Grassmannian $w\in \wt{W}_3$}
\centering
\begin{tabular}{|c|c|c|c|} \hline
$w$ & $\la $ & $\beta$ & $H_\la(c\, |\, t)$
\\ \hline

$123$ &  &  & $1$ \\

$213$ & $1$ & $(1,3)$ & $b_1+h_1^1$ \\
 
$\ov{2}\ov{1}3$ & $1'$ & $(0,3)$ & $\wt{b}_1$ \\

$\ov{1}\ov{2}3$ & $2$ & $(-1,3)$ & $b_2+\wt{b}_1e_1^1$ \\

$312$ & $(1,1)$ & $(1,2)$ &  
$(b_1+h_1^1)(c_1+h_1^2) - (b_2+c_1h_1^1+h_2^1)$ \\

$\ov{3}\ov{1}2$ & $(1,1)'$ & $(0,2)$ &  
$\wt{b}_1(c_1+h_1^2) - b_2$ \\

$\ov{1}\ov{3}2$ & $3$ & $(-2,2)$ & $b_3+b_2e_1^2+\wt{b}_1e_2^2$ \\

$3\ov{2}\ov{1}$ & $(2,1)$ & $(-1,0)$  &  
$(b_2+\wt{b}_1e_1^1)b_1-(b_3+b_2e_1^1)$ \\

$\ov{3}\ov{2}1$ & $(2,1)'$  & $(-1,1)$  &  
$(b_2+\wt{b}_1e_1^1)(\wt{b}_1+h_1^1)-(b_3+b_2e_1^1)$  \\

$2\ov{3}\ov{1}$ & $(3,1)$ & $(-2,0)$ &  
$(b_3+b_2e_1^2+\wt{b}_1e_2^2)b_1-(b_4+b_3e_1^2+b_2e_2^2)$ \\

$\ov{2}\ov{3}1$ & $(3,1)'$ &  $(-2,1)$ & 
$(b_3+b_2e_1^2+\wt{b}_1e_2^2)(\wt{b}_1+h_1^1)-(b_4+b_3e_1^2+b_2e_2^2)$ \\

$1\ov{3}\ov{2}$ & $(3,2)$  &  $(-2,-1)$ & 
$(b_3+b_2e_1^2+\wt{b}_1e_2^2)(b_2+b_1e_1^1)$ \\

&&& $-(b_4+b_3e_1^2+b_2e_2^2)(c_1+e_1^1)+(b_5+b_4e_1^2+b_3e_2^2)$ \\

\hline

$132$  &  $1$  &  $2$  &  $b_1+h_1^2$ \\

$231$  &  $2$  &  $1$  &  $b_2+b_1h_1^1+h_2^1$ \\

$\ov{2}3\ov{1}$  &  $2'$  &  $0$  &  $\wt{b}_2$ \\

$\ov{1}3\ov{2}$  &  $3$  &  $-1$  &  $b_3+\wt{b}_2e_1^1$ \\

$\ov{1}2\ov{3}$  &  $4$  &  $-2$  &  $b_4+b_3e_1^2+\wt{b}_2e_2^2$ \\

\hline
\end{tabular}
\label{table1}
\end{table}}}

Let $\{e_1,\ldots,e_{2n}\}$ denote the standard orthogonal basis of
$\C^{2n}$ and let $F_i$ be the subspace spanned by the first $i$
vectors of this basis, so that $F_{n-i}^\perp = F_{n+i}$ for $0\leq i
\leq n$. Let $B_n$ denote the stabilizer of the flag $F_\bull$ in the
group $\SO_{2n}(\C)$, and let $T_n$ be the corresponding maximal torus
in the Borel subgroup $B_n$. The $T_n$-equivariant cohomology ring
$\HH^*_{T_n}(\OG(n-k,2n),\Z)$ is defined as the cohomology ring of the
Borel mixing space $ET_n\times^{T_n}\OG$. The Schubert cells in $\OG$
are the $B_n$-orbits, and are indexed by the typed $k$-strict
partitions $\la$ whose Young diagram fits in an $(n-k)\times (n+k-1)$
rectangle. Any such $\la$ defines a Schubert cell
$X^\circ_\la=X^{\circ}_\la(F_\bull)$ of codimension
$|\la|:=\sum_i\la_i$ by the prescription
\[
   X^\circ_\lambda := \{ \Sigma \in \OG \mid \dim(\Sigma \cap
   F_q) = \# \{j \, |\, p_j(\la)\leq q\} \ \ \ \forall\, 1 \leq q \leq
   2n \}
\]
where, for $1\leq j \leq n-k$, we have
\[
p_j(\la):= n+\begin{cases} 
1-\be_j(\la) & \text{if $\be_j(\la) \in\{0,1\}$ and $n$ is odd}, \\
\be_j(\la) & \text{otherwise}. 
\end{cases}
\]
The Schubert variety $X_\la$ is the closure of the Schubert cell
$X^\circ_\la$.  Since $X_\la$ is stable under the action of $T_n$, we
obtain an {\em equivariant Schubert class}
$[X_\la]^{T_n}:=[ET_n\times^{T_n}X_\la]$ in
$\HH^*_{T_n}(\OG(n-k,2n))$.

The natural inclusions $\wt{W}_n\hookrightarrow
\wt{W}_{n+1}$ of the Weyl groups defined earlier induce 
surjections of graded algebras
\[
\cdots \rightarrow \HH^*_{T_{n+1}}(\OG(n+1-k,2n+2)) \rightarrow
\HH^*_{T_n}(\OG(n-k,2n))\rightarrow \cdots
\]
and the stable equivariant cohomology ring of $\OG$, denoted by
$\IH_T(\OG_k)$, is the associated graded inverse limit
\[
\IH_T(\OG_k) := \lim_{\longleftarrow}\HH_{T_n}^*(\OG(n-k,2n)).
\]
One identifies here the variables $t_i$ with the characters of the
maximal tori $T_n$ in a compatible way, as in \cite[\S 2]{BH} and
\cite[\S 10]{IMN1}. We then have that $\IH_T(\OG_k)$ is a free
$\Z[t]$-algebra with a basis of stable equivariant Schubert classes
$$\tau_\la:=\lim_{\longleftarrow}[X_\la]^{T_n},$$ one for every typed
$k$-strict partition $\la$.

Consider the graded polynomial ring
$\Z[b]:=\Z[\wt{b}_k,b_1,b_2,\ldots]$, where the variable $b_i$ has
degree $i$ for each $i$, and $\wt{b}_k$ has degree $k$. 
Let $J^{(k)}\subset \Z[b]$ be the homogeneous ideal generated by the
relations
\begin{gather}
\label{relation1}
b_pb_p + \sum_{i=1}^p(-1)^i b_{p+i}c_{p-i}= 0
\ \ \ \text{for} \  p > k, \\
\label{relation2}
b_k\wt{b}_k+\sum_{i=1}^k(-1)^i b_{k+i}b_{k-i} = 0,
\end{gather}
where the $c_i$ satisfy the relations (\ref{ctob}), 
and define the quotient ring $B^{(k)}:=\Z[b]/{J^{(k)}}$.  We call
the graded polynomial ring $B^{(k)}[t]$ the {\em ring of double eta
polynomials}. 

The following result establishes the precise connection between the
double eta polynomials $\Eta_\la(c\, |\, t)$ and the equivariant
Schubert classes on $\OG$, namely, that the former represent the
latter. We regard $\HH_{T_n}^*(\OG(n-k,2n))$ as a
$\Z[t]$-module via the natural projection map $\Z[t]\to
\Z[t_1,\ldots,t_n]$.

\begin{thm}
\label{mainthm} 
The polynomials $\Eta_\la(c\, |\, t)$, as $\la$ runs
over all typed $k$-strict partitions, form a free $\Z[t]$-basis of
$B^{(k)}[t]$. There is an isomorphism of graded $\Z[t]$-algebras
\[
\pi: B^{(k)}[t]\to \IH_T(\OG_k)
\]
such that $\Eta_\la(c\, |\, t)$ is mapped to $\tau_\la$ for every
typed $k$-strict partition $\la$. For every $n\geq 1$, the morphism
$\pi$ induces a surjective homomorphism of graded $\Z[t]$-algebras
$$\pi_n: B^{(k)}[t]\to \HH_{T_n}^*(\OG(n-k,2n))$$ which maps
$\Eta_\la(c\, |\, t)$ to $[X_\la]^{T_n}$, if $\la$ fits inside an
$(n-k)\times (n+k-1)$ rectangle, and to zero, otherwise. 
\end{thm}

The map $\pi_n$ in Theorem \ref{mainthm} is induced from the type D
geometrization map of \cite[\S 10]{IMN1} and \cite[\S 7]{T2} (see \S
\ref{geomapD}). When all the parts $\la_i$ of the indexing typed
$k$-strict partition $\la$ are greater than $k$, then the equality
$[X_\la]^{T_n} = \pi_n(\Eta_\la(c\, |\, t))$ is equivalent to the
Chern class formula for even orthogonal degeneracy loci obtained by
Kazarian \cite{Ka} in 2001. When we set $t=0$, Theorem \ref{mainthm}
gives the Giambelli formula for the ordinary Schubert classes on $\OG$
from \cite[Thm.\ 1]{BKT2}.

Our proof of Theorem \ref{mainthm} follows the argument of \cite{TW},
which dealt with the analogous theory of double theta polynomials
$\Ti_\la(c\, |\, t)$ for the symplectic Grassmannians. Adapting the
work of Ikeda and Matsumura \cite{IM}, we showed in \cite[\S 5]{TW}
that the $\Ti_\la(c\, |\, t)$ are compatible with the action of left
divided difference operators on the polynomial ring $\Z[c,t]$.  In the
type D framework of the present paper, we similarly prove that the
action of $\wt{W}_\infty$ on $B^{(k)}[t]$ lifts to an action on
$\Z[b,t]$, and gives rise to divided differences there. In \S
\ref{Hhat}, we introduce a family of double polynomials
$\wh{\Eta}_\la(c\, |\, t)$ which are indexed by $k$-strict
partitions. These specialize to the single polynomials
$\wh{\Eta}_\la(c)$ of \cite[\S 5.2]{BKT2}, are compatible with the
divided differences on $\Z[b,t]$, and enjoy properties entirely
parallel to those of the double theta polynomials $\Ti_\la(c\, |\,
t)$. However, the double eta polynomials $\Eta_\la(c\, |\, t)$ are
more subtle: there are instances where the compatibility with divided
differences is true for them only modulo the relation
(\ref{relation2}) (see Proposition \ref{uniq} and Remark
\ref{rm1}). We conclude the proof Theorem \ref{mainthm} by using a
formula for the equivariant Schubert class of a point, which is a
special case of the aforementioned result from \cite{Ka}.

It is important to note that the double eta polynomials $\Eta_\la(c\,
|\, t)$ defined here are new, and are not equal to the type D double
Schubert polynomials of \cite{IMN1} indexed by the $k$-Grassmannian
elements of the Weyl group, which represent equivariant Schubert
classes on the complete even orthogonal flag variety.  The latter
objects are really formal power series, and are expressed using a
different set of variables which are not intrinsic to the Grassmannian
$\OG(n-k,2n)$. The precise relationship between the two familes of
polynomials is discussed in \S \ref{geomapD}.

Our research on this article was influenced by three prior works:
Kazarian's paper \cite{Ka} on degeneracy locus formulas of Pfaffian
type, Wilson's thesis \cite{W}, where double theta polynomials were
first defined and studied, and Ikeda and Matsumura's article
\cite{IM}, which exhibited the compatibility of these polynomials with
left divided differences, and proved that they represent equivariant
Schubert classes. We thank each of these authors for their
contributions. In recent work, Anderson and Fulton \cite{AF} have
defined a family of double eta polynomials independently, and extended
them further to `multi-eta polynomials', which represent (a power of 2
times) the classes of certain degeneracy loci of even orthogonal type.

This paper is organized as follows. In Section \ref{prelim}, we define
the type D divided difference operators on $\Z[b,t]$ and establish
their basic properties. Section \ref{etasec} proves the required
compatibility of double eta polynomials with divided differences, and
studies the related family of polynomials $\wh{\Eta}_\la(c\, |\,
t)$. The proof of Theorem \ref{mainthm} is completed in Section
\ref{gddsD}, which also describes how the polynomials $\Eta_\la(c\,
|\, t)$ are related to the general equivariant Giambelli polynomials
of \cite{T1}.

Our work on double eta polynomials was announced during the conference
`IMPANGA 15' which took place in B\k{e}dlewo, Poland. It is a pleasure
to thank the organizers for their hospitality and for making this
stimulating event possible. I also thank the referee for comments
which helped to improve the exposition.

\section{Divided difference operators on $\Z[b,t]$}
\label{prelim}

In this section we will work exclusively in the polynomial ring
$\Z[b,t]$.  We begin by defining the action of the Weyl group
$\wt{W}_\infty$ on $\Z[b,t]$ by ring
automorphisms and the associated family of $t$-divided
difference operators $\{\partial_i\}_{i\geq 0}$ on $\Z[b,t]$.

The elements of the Weyl group $\wt{W}_n$ of type $\mathrm{D}_n$ are
represented as signed permutations of the set $\{1,\ldots,n\}$ with an
even number of negative entries. The group $\wt{W}_n$ is generated by
the simple transpositions $s_i=(i,i+1)$ for $1\leq i \leq n-1$ and an
element $s_0:=s_0^Bs_1s_0^B$, where $s_0^B(1)=\ov{1}$ denotes the sign
change. The action of $\wt{W}_\infty$ on $\Z[b,t]$ is defined as
follows. The simple reflections $s_i$ for $i>0$ act by interchanging
$t_i$ and $t_{i+1}$ and leaving all the remaining variables fixed. The
reflection $s_0$ maps $(t_1,t_2)$ to $(-t_2,-t_1)$, fixes the $t_j$
for $j\geq 3$, and satisfies the equations
\[
s_0(b_p) := \begin{cases}
b_p-2(t_1+t_2)c_{p-1}^2 & \text{if $p<k$}, \\
b_p-(t_1+t_2)c_{p-1}^2& \text{if $p \geq k$}
\end{cases}
\]
and $s_0(\wt{b}_k) := \wt{b}_k-(t_1+t_2)c_{k-1}^2$.
Observe that for every $p\geq 1$, we have
\[
s_0(c_p) = c_p-2(t_1+t_2)c_{p-1}^2 =
c_p-2(t_1+t_2)\sum_{j=0}^{p-1}(-1)^j\left(\sum_{a+b=j} t_1^a
t_2^b\right)c_{p-1-j}.
\]
It is useful to write this as an equation of generating series 
\begin{equation}
\label{beq}
s_0\left(\sum_{p=0}^{\infty}c_pu^p\right) = 
\frac{1-t_1u}{1+t_1u}\frac{1-t_2u}{1+t_2u}\cdot\sum_{p=0}^{\infty}c_pu^p
\end{equation}
where $u$ denotes a formal variable such that $s_i(u)=u$ for each $i$. 

One checks that, with the above definition of $s_i$ for $i\geq 0$, the
braid relations for $\wt{W}_{\infty}$ are satisfied in $\Z[b,t]$, and
so we obtain a well defined group action. Moreover, the action of
$\wt{W}_{\infty}$ on $\Z[b,t]$ induces an action on the quotient ring
$B^{(k)}[t]$.  Define the {\em divided difference operators}
$\partial_i$ on $\Z[b,t]$ by
\[
\partial_0 f := \frac{f-s_0 f}{t_1+t_2}, \qquad
\partial_if := \frac{f-s_if}{t_{i+1}-t_i}, \ \ \ \ \text{if 
$i\geq 1$}.
\]
The same equations also define operators $\partial_i$ on
$B^{(k)}[t]$. These latter correspond to the left divided
differences $\delta_i$ studied in \cite{IMN1}. The previous formulas
imply that
\[
\partial_0(c_p) = 2c^2_{p-1} \ \ \ \mathrm{and} \ \ \ 
\partial_0(b_k) = \partial_0(\wt{b}_k) = c^2_{k-1}.
\]
For every $i\geq 0$, the operator $\partial_i$ satisfies the Leibnitz
rule
\[
\partial_i(fg) = (\partial_if)g+(s_if)\partial_ig.
\]

For $r<0$ we let $t_{-r}:=t_r$. We recall the following basic result 
from \cite[\S 1]{TW}.

\begin{lemma}
\label{ctlem1}
Suppose that $p,r\in \Z$.

\medskip
\noin
{\em (a)} Assume that $r>0$. Then we have
\[
c_p^r = c_p^{r-1} - t_r\, c_{p-1}^r.
\]

\medskip
\noin
{\em (b)} Assume that $r \leq 0$. Then we have
\[
c_p^r = c_p^{r-1} + t_{r-1}\, c_{p-1}^r.
\]
\end{lemma}

We now prove several identities satisfied by the operators $\partial_i$,
analogous to those shown in \cite[\S 5]{TW}. Observe first that, for
$r\geq 1$, we have
\begin{equation}
\label{generating}
\sum_{p=0}^{\infty}c_p^ru^p = \left(\sum_{i=0}^{\infty}c_iu^i\right)\prod_{j=1}^r
\frac{1}{1+t_ju}
\end{equation}
while, for $r \leq -1$, we have 
\begin{equation}
\label{generating2}
\sum_{p=0}^{\infty}c_p^ru^p =
\left(\sum_{i=0}^{\infty}c_iu^i\right)\prod_{j=1}^{|r|}(1-t_ju).
\end{equation}

\begin{lemma}
\label{lem1}
Suppose that $i\geq 1$. We have the identities
\[
s_i(c_p^r) = 
\begin{cases}
c_p^r & \text{if $r \neq \pm i$}, \\
c_p^{i+1}+t_ic_{p-1}^{i+1} & \text{if $r= i > 0$}, \\
c_p^{-i+1}-t_{i+1}c_{p-1}^{-i+1} & \text{if $r= -i < 0$}
\end{cases}
\]
and 
\[
s_0(c_p^r) = 
\begin{cases}
c_p^r & \text{if $|r| \geq 2$}, \\
c_p^2-t_1c_{p-1}^2 & \text{if $r= 1$}, \\
c_p^2 - (t_1+t_2)c_{p-1}^2 + t_1t_2c_{p-2}^2 & \text{if $r= 0$}, \\
c_p^1 - (t_1+t_2)c_{p-1}^1 + t_1t_2c_{p-2}^1 & \text{if $r=-1$}.
\end{cases}
\]
\end{lemma}
\begin{proof}
Since $c_p^r$ is symmetric in $(t_1,\ldots,t_{|r|})$, the identity
$s_i(c_p^r) = c_p^r$ for $r\neq \pm i$ is clear. If $r\geq 2$,
then we apply $s_0$ to both sides of (\ref{generating})
and use (\ref{beq}) to show that 
$s_0(c_p^r)=c_p^r$ for all $p$; the proof when $r\leq -2$ is similar,
using (\ref{generating2}).

If $r= i > 0$, then $s_i(c_p^i) = c_p^{i+1}+t_ic_{p-1}^{i+1}$ follows from 
the identity
\[
s_i\left(\frac{1}{1+t_i}\right) =
\frac{1}{1+t_{i+1}} = \frac{1}{(1+t_i)(1+t_{i+1})} + 
\frac{t_i}{(1+t_i)(1+t_{i+1})}.
\]
If $r= -i < 0$, then $s_i(c_p^{-i}) = c_p^{-i+1}-t_{i+1}c_{p-1}^{-i+1}$ follows from 
$s_i(1-t_i) = 1-t_{i+1}$. 

If $r=1$ then equation (\ref{generating}) gives
\[
s_0\left(\sum_{p=0}^{\infty}c_p^1u^p\right) = \frac{1-t_1u}{(1+t_1u)(1+t_2u)}
\left(\sum_{p=0}^{\infty}c_pu^p\right) =
(1-t_1u)\left(\sum_{p=0}^{\infty}c^2_pu^p\right)
\]
while if $r=-1$, equation (\ref{generating2}) gives
\[
s_0\left(\sum_{p=0}^{\infty}c_p^{-1}u^p\right) = \frac{(1-t_1u)(1-t_2u)}{1+t_1u}
\left(\sum_{p=0}^{\infty}c_pu^p\right) = 
(1-t_1u)(1-t_2u)\left(\sum_{p=0}^{\infty}c^1_pu^p\right).
\]
The displayed formulas for $s_0(c_p^1)$ and $s_0(c_p^{-1})$
follow. Finally, we use equation (\ref{beq}) to compute 
$s_0(c_p)$.
\end{proof}

\begin{prop}
\label{prop1}
Suppose that $p,r\in \Z$.

\medskip
\noin
{\em (a)} For all $i \geq 1$, we have 
\[
\partial_ic_p^r= 
\begin{cases}
c_{p-1}^{r+1} & \text{if $r=\pm i$}, \\
0 & \text{otherwise}.
\end{cases}
\]
We have 
\[
\partial_0 c_p^r= 
\begin{cases}
c_{p-1}^2 & \text{if $r=1$}, \\
2c_{p-1}^2 & \text{if $r=0$}, \\
2c_{p-1}^1 -c_{p-1} & \text{if $r=-1$}, \\
0 & \text{if $|r|\geq 2$}.
\end{cases}
\]
In particular, we have
\begin{equation}
\label{beqN}
\partial_0 c_p^{-1} = 2a_{p-1}^1, \ \ \ 
\partial_1 c_p^{-1} = 2a_{p-1}^0, \ \ \text{and} \ \ 
(\partial_0+\partial_1)c_p^{-1} = 2c_{p-1}^1.
\end{equation}

\medskip
\noin
{\em (b)} For all $i\geq 1$, we have 
\begin{equation}
\label{iieqN}
\partial_i(c_p^{-i}c_q^i) = c_{p-1}^{-i+1}c_q^{i+1} +
c_p^{-i+1}c_{q-1}^{i+1}.
\end{equation}

\noin
{\em (c)} We have 
\begin{gather}
\label{12eqA}
\partial_0(c_p^{-1}c_q^1) = 2(a_{p-1}^1c_q^2+a_p^1c_{q-1}^2), \ \ 
\partial_1(c_p^{-1}c_q^1) = 2(a_{p-1}^0c_q^2+a_p^0c_{q-1}^2), \ \ \text{and} \\
\label{12eqB}
(\partial_0+\partial_1)(c_p^{-1}c_q^1) = 2(c_{p-1}^1c_q^2 + c_p^1c_{q-1}^2).
\end{gather}
\end{prop}
\begin{proof}
For part (a), observe that if $i>0$ and $r \neq \pm i$, then the 
result follows from Lemma \ref{lem1} immediately. If $r=i>0$, then we 
compute easily that
\[
\partial_i\left(\sum_{p=0}^\infty c_p^r u^p \right) = 
\left(\sum_{p=0}^\infty c_p u^{p+1} \right)\prod_{j=1}^{r+1}\frac{1}{1+t_ju}
\]
from which the desired result follows. We work similarly when $r=-i<0$. 

To evaluate the divided difference $\partial_0$, note, for instance, that 
\[
\sum_{p=0}^{\infty}c^1_pu^p - s_0\left(\sum_{p=0}^{\infty}c^1_pu^p\right) =
\frac{(t_1+t_2)u}{(1+t_1u)(1+t_2u)}\left(\sum_{p=0}^{\infty}c_pu^p\right)
\]
and the computation of $\partial_0(c_p^1)$ follows. The rest
are evaluated similarly, but we pay special attention to the third case.
We compute using the Leibnitz rule that
\[
\partial_0(c_p^{-1}) = \partial_0(c_p-t_1c_{p-1}) = 
2c_{p-1}^2 - c_{p-1} +2t_2c^2_{p-2}.
\]
However, for any $p$, we have
\[
c^2_p+t_2c^2_{p-1} = c_p+\sum_{j=0}^{p-1}c_{p-1-j}(h^2_{j+1}(-t)+t_2h^2_j(-t))
= c_p+\sum_{j=0}^{p-1}c_{p-1-j}h^1_{j+1}(-t) = c^1_p.
\]
It follows that $\partial_0(c_p^{-1}) = 2c^1_{p-1}-c_{p-1} = 2a_{p-1}^1$. Since 
$\partial_1(c_p^{-1}) = c_{p-1} = 2a_{p-1}^0$, the equations (\ref{beqN}) follow.

For part (b), use the Leibnitz rule and Lemmas \ref{ctlem1} and \ref{lem1}
to compute
\begin{align*}
\partial_i(c_p^{-i}c_q^i) &= \partial_i(c_p^{-i})c_q^i + s_i(c_p^{-i}) 
\partial_i(c_q^i) \\
&= c_{p-1}^{-i+1}c_q^i + (c_p^{-i+1}-t_{i+1}c_{p-1}^{-i+1})c_{q-1}^{i+1} \\
&= c_{p-1}^{-i+1}(c_q^i - t_{i+1}c_{q-1}^{i+1}) + c_p^{-i+1}c_{q-1}^{i+1} \\
&= c_{p-1}^{-i+1}c_q^{i+1} + c_p^{-i+1}c_{q-1}^{i+1}.
\end{align*}
We finally establish the equations (\ref{12eqA}) and (\ref{12eqB}). 
An application of (\ref{iieqN}) gives
\[
\partial_1(c_p^{-1}c_q^1) = c_{p-1}c_q^2 +
c_pc_{q-1}^2 = 2(a_{p-1}^0c_q^2+a_p^0c_{q-1}^2).
\]
The Leibnitz rule implies that
\begin{align*}
\partial_0(c_p^{-1}c_q^1) &= 2a^1_{p-1}c_q^1 +
s_0(c^{-1}_p)c_{q-1}^2 \\
&= 2a^1_{p-1}(c^2_q+t_2c^2_{q-1}) + 
(c^1_p-(t_1+t_2)c^1_{p-1}+t_1t_2c^1_{p-2})c^2_{q-1} \\
&= 2a^1_{p-1}c^2_q+(c^1_p-t_1c^1_{p-1}+t_2(c^1_{p-1}
-c_{p-1}+t_1c^1_{p-2}))c^2_{q-1} \\
&= 2a^1_{p-1}c^2_q+(2c^1_p-c_p)c^2_{q-1} = 2(a_{p-1}^1c_q^2+a_p^1c_{q-1}^2).
\end{align*}
where we employed the identity $c_q^1=c^2_q+t_2c^2_{q-1}$ and, in the
second to last equation, the identity $c^1_p=c_p-t_1c^1_{p-1}$
twice. This completes the proof.
\end{proof}

We will require certain variations of the previous identities. Set
$a_p:=a_p^0=\frac{1}{2}c_p$ for each integer $p$. Let $f_k$ be a
variable of degree $k$, which will be equal to $b_k$, $\wt{b}_k$, or
$a_k$, depending on the context. For $s\in \{0,1\}$, define
\[
f_k^s:= f_k+\sum_{j=1}^kc_{k-j}h_j^s(-t), 
\]
set $\wt{f}_k:=c_k-f_k$, and $\wt{f}_k^s:=c_k-2f_k+f_k^s$.
For any $p,r\in \Z$, define $\wh{c}_p^r$ by 
\[
\wh{c}_p^r:= c_p^r + 
\begin{cases}
(2f_k-c_k)e^{p-k}_{p-k}(-t) & \text{if $r = k - p < 0$}, \\
0 & \text{otherwise}.
\end{cases}
\]

\begin{lemma}
\label{lem11N}
Suppose that $i\geq 1$. We then have the identities
\[
s_i(\wh{c}_p^r) = 
\begin{cases}
\wh{c}_p^r & \text{if $r \neq \pm i$}, \\
\wh{c}_p^{i+1}+t_i\wh{c}_{p-1}^{i+1} & \text{if $r= i > 0$}, \\
\wh{c}_p^{-i+1}-t_{i+1}\wh{c}_{p-1}^{-i+1} & \text{if $r= -i < 0$}.
\end{cases}
\]
Moreover, if $|r|\geq 2$, then $s_0(\wh{c}_p^r) = \wh{c}_p^r$.
\end{lemma}
\begin{proof}
If $p>k$, then we have that
\[
\wh{c}_p^{k-p} = c_p^{k-p}+(2f_k-c_k)e^{p-k}_{p-k}(-t). 
\]
Since $s_i((2f_k-c_k)e^{p-k}_{p-k}(-t))= (2f_k-c_k)s_i(e^{p-k}_{p-k}(-t))$,
the identity $s_i(\wh{c}_p^{p-k}) = \wh{c}_p^{p-k}$ for ${p-k}\neq \pm i$ is
clear.  If $p-k= -i$, then $s_i(\wh{c}_p^{-i}) =
\wh{c}_p^{-i+1}-t_{i+1}\wh{c}_{p-1}^{-i+1}$ follows from the
corresponding identity of Lemma \ref{lem1} and the calculation
\[
s_i(e^i_{p-k}(-t))= e^{i-1}_{p-k}(-t)-t_{i+1}e^{i-1}_{p-1-k}(-t).
\]
We have $\wh{c}_p^{k-p} = c_p^{k-p}+(-1)^{p-k}(2f_k-c_k)t_1\cdots
t_{p-k}$. If $p-k\geq 2$, then $s_0$ leaves all terms on the right
hand side invariant, and hence $s_0(\wh{c}_p^{k-p}) =
\wh{c}_p^{k-p}$. The remaining equalities follow from Lemma \ref{lem1}.
\end{proof}

\begin{prop}
\label{prop1hN}
Suppose that $p\in \Z$ and $p>k$.

\medskip
\noin
{\em (a)} For all $i \geq 1$, we have 
\[
\partial_i\wh{c}_p^{k-p}= 
\begin{cases}
  \wh{c}_{p-1}^{k-p+1} & \text{if $i=p-k\geq 2$}, \\
  2f_k & \text{if $i=p-k= 1$}, \\
0 & \text{otherwise}.
\end{cases}
\]
We have 
\[
\partial_0 \wh{c}_p^{k-p}= 
\begin{cases}
2\wt{f}^1_k & \text{if $k-p=-1$}, \\
0 & \text{if $k-p < -1$}.
\end{cases}
\]
In particular, we have
\begin{equation}
\label{beqhN}
\partial_0 \wh{c}_{k+1}^{-1} = 2\wt{f}^1_k, \ \ \ 
\partial_1 \wh{c}_{k+1}^{-1} = 2f_k, \ \ \text{and} \ \ 
(\partial_0+\partial_1)\wh{c}_{k+1}^{-1} = 2c_k^1.
\end{equation}

\medskip
\noin
{\em (b)} For all $i\geq 2$, we have 
\begin{equation}
\label{iieqhN}
\partial_i(\wh{c}_p^{-i}c_q^i) = \wh{c}_{p-1}^{-i+1}c_q^{i+1} +
\wh{c}_p^{-i+1}c_{q-1}^{i+1}.
\end{equation}

\noin
{\em (c)} We have 
\begin{gather}
\label{12eqANh}
\partial_0(\wh{c}_{k+1}^{-1} c_q^1) =
2(\wt{f}^1_kc_q^2+a_{k+1}^1 c_{q-1}^2), \\
\label{12eqANNh}
\partial_1(\wh{c}_{k+1}^{-1} c_q^1) =
2(f_k c_q^2+a_{k+1}^0 c_{q-1}^2), \ \ \text{and} \\
\label{12eqBNh}
(\partial_0+\partial_1)(\wh{c}_{k+1}^{-1}c_q^1) =
2(c_k^1c_q^2 + c_{k+1}^1 c_{q-1}^2).
\end{gather}
\end{prop}
\begin{proof}
Recall that 
\[
\wh{c}_p^{k-p} = c_p^{k-p}+(2f_k-c_k)e^{p-k}_{p-k}(-t).
\]

For part (a), observe that if  $p-k \neq  i\geq 1$, then the 
result follows from Lemma \ref{lem1}. If $i=p-k\geq 2$, then we 
compute that 
\[
\partial_i\left((2f_k-c_k)e^{p-k}_{p-k}(-t)\right) = 
(2f_k-c_k)\partial_ie_{p-k}^i(-t) = (2f_k-c_k)e^{i-1}_{p-1-k}(-t)
\]
from which the desired result follows. For $p=k+1$, we have
\[
\wh{c}^{-1}_{k+1}= c^{-1}_{k+1}+t_1(c_k-2f_k)
\]
and we compute that
\[
\partial_1(\wh{c}^{-1}_{k+1}) = c_k- (c_k-2f_k) = 2f_k.
\]
The fact that
$\partial_0 \wh{c}_p^{k-p} = 0$ for $k-p < -1$ 
follows immediately from Lemma \ref{lem11N}. 
Since Proposition \ref{prop1}(a) gives $\partial_0(c^{-1}_{k+1}) = 
2c^1_k-c_k$, we deduce that 
\begin{equation}
\label{cmin}
\partial_0(\wh{c}^{-1}_{k+1}) = 2c^1_k-2f_k = 2\wt{f}^1_k.
\end{equation}
This completes the proof of part (a).

Part (b) follows from the Leibnitz rule and Lemmas \ref{ctlem1} and 
\ref{lem11N}, exactly as in the proof of (\ref{iieqN}). 
For part (c), use (\ref{12eqA}) to compute that
\begin{align*}
\partial_1(\wh{c}_{k+1}^{-1} c_q^1) &= 
\partial_1(c_{k+1}^{-1}c_q^1 +(c_k-2f_k)t_1c^1_q) \\
&= \partial_1(c_{k+1}^{-1}c_q^1) + 
(c_k-2f_k)\partial_1(t_1c_q^1) \\
&= (c_kc_q^2+c_{k+1}c^2_{q-1})-(c_k-2f_k)c^2_q \\
&=  2f_kc_q^2+c_{k+1}c^2_{q-1}.
\end{align*}
We similarly have
\begin{align*}
\partial_0(\wh{c}_{k+1}^{-1} c_q^1) &= 
\partial_0(c_{k+1}^{-1}c_q^1 +(c_k-2f_k)t_1c^1_q) \\
&= \partial_0(c_{k+1}^{-1}c_q^1) + 
(c_k-2f_k)\partial_0(t_1c_q^1) \\
&= 2(a^1_kc_q^2+a^1_{k+1}c^2_{q-1})+(c_k-2f_k)c^2_q \\
&= (2c^1_k-2f_k)c_q^2+2a^1_{k+1}c^2_{q-1}.
\end{align*}
\end{proof}

\begin{prop}
\label{proprel}
Suppose that $p,q\in \Z$. We have
\begin{gather}
\label{12equ2}
\partial_0(c_p^{-1}a_q^0) =
2(a^1_{p-1}c^2_q+a^1_pc^2_{q-1})- 2a^1_{p-1}a^1_q \\
\partial_0(c_p^{-1}\wt{f}_k^0) =
2(a^1_{p-1}c^2_k+a^1_pc^2_{k-1})-2a^1_{p-1}f^1_k \\
\partial_0(\wh{c}_{k+1}^{-1}a_q^0) =
2(\wt{f}^1_kc_q^2+a_{k+1}^1 c_{q-1}^2)-2\wt{f}^1_ka^1_q, \ \ \text{and} \\
\label{12Bequ2}
\partial_0(\wh{c}_{k+1}^{-1}\wt{f}_k^0) =
2(\wt{f}^1_kc^2_k+a^1_{k+1}c^2_{k-1})-2\wt{f}^1_kf^1_k.
\end{gather}
We also have 
\begin{gather}
\label{12equ}
\partial_1(c_p^{-1}a_q^1) =
2(a^0_{p-1}c^2_q+a^0_pc^2_{q-1})-2a_{p-1}a_q \\
\partial_1(c_p^{-1}f_k^1) =
2(a^0_{p-1}c^2_k+a^0_pc^2_{k-1})-2a_{p-1}\wt{f}_k \\
\partial_1(\wh{c}_{k+1}^{-1}a_q^1) =
2(f^0_kc_q^2+a_{k+1}^0 c_{q-1}^2)-2f_ka_q, \ \ \text{and} \\
\label{12Bequ}
\partial_1(\wh{c}_{k+1}^{-1}f_k^1) =
2(f^0_kc^2_k+a^0_{k+1}c^2_{k-1})-2f_k\wt{f}_k.
\end{gather}
\end{prop}
\begin{proof}
To prove equation (\ref{12equ2}), we use (\ref{12eqA}), the Leibnitz
rule, and the observation that $\partial_0(a^1_q)=0$, to compute
\[
\partial_0(c_p^{-1}a_q^0) = \partial_0(c_p^{-1}c_q^1) -
\partial_0(c_p^{-1}a_q^1) = 2(a_{p-1}^1c_q^2+a_p^1c_{q-1}^2) - 2 a_{p-1}^1a^1_q.
\]
Since
$\partial_0(f^1_k)=\partial_0(c^1_k-\wt{f}_k)=c^2_{k-1}-c^2_{k-1}=0$,
we similarly have that
\[
\partial_0(c_p^{-1}\wt{f}^0_k) = \partial_0(c_p^{-1}c_k^1) -
\partial_0(c_p^{-1} f^1_k)  
 = 2(a^1_{p-1}c^2_k+a^1_pc^2_{k-1})-2a^1_{p-1}f^1_k.
\]
We compute using (\ref{12eqANh}) that
\[
\partial_0(\wh{c}_{k+1}^{-1}a_q^0) = 
\partial_0(\wh{c}_{k+1}^{-1}c_q^1) -
\partial_0(\wh{c}_{k+1}^{-1}a_q^1)  
 = 2(\wt{f}^1_kc_q^2+a_{k+1}^1 c_{q-1}^2)-2\wt{f}^1_ka^1_q.
\]
We similarly have
\[
\partial_0(\wh{c}_{k+1}^{-1}\wt{f}^0_k) = \partial_0(\wh{c}_{k+1}^{-1}c_k^1) -
\partial_0(\wh{c}_{k+1}^{-1}f^1_k)  
 = 2(\wt{f}^1_kc^2_k+a^1_{k+1}c^2_{k-1})-2\wt{f}^1_kf^1_k.
\]
The proof of equations (\ref{12equ})--(\ref{12Bequ}) is analogous,
applying (\ref{12eqA}) and (\ref{12eqANNh}).
\end{proof}

\section{Double eta polynomials}
\label{etasec}

\subsection{A basis theorem}
For the rest of this paper, we will sometimes write equalities that
hold only in the ring $B^{(k)}[t]$, where we have imposed the
relations (\ref{relation1}) and (\ref{relation2}) on the generators
$b_p$. Whenever these relations are needed, we will emphasize this by
noting that the equalities are true in $B^{(k)}[t]$ rather than in
$\Z[b,t]$.

We begin this section with a basis theorem for the $\Z[t]$-algebra 
$B^{(k)}[t]$. For any typed $k$-strict partition $\la$ of length $\ell$,
with $m:=\ell_k(\la)+1$, we define $b_\la\in \Z[b]$ as follows. 
If $\type(\la)=0$, then set $c_\la:=c_{\la_1}\cdots c_{\la_\ell}$, while 
if $\type(\la)>0$, define
\[
c_{\la} := \begin{cases}
c_{\la_1} \cdots c_{\la_{m-1}} \, b_k \, c_{\la_{m+1}} \cdots c_{\la_\ell} & 
\text{if  $\,\type(\la) = 1$}, \\
c_{\la_1} \cdots c_{\la_{m-1}} \, \wt{b}_k \, c_{\la_{m+1}} \cdots c_{\la_\ell} & 
\text{if  $\,\type(\la) = 2$}.
\end{cases}
\]
Finally, define $b_\la:= 2^{-\ell_k(\la)} \, c_\la$. 

\begin{prop}
\label{basisthm}
The monomials $b_\la$, the single eta polynomials $\Eta_\la(c)$, and 
the double eta polynomials $\Eta_\la(c\, |\, t)$ form three 
$\Z[t]$-bases of $B^{(k)}[t]$, as $\la$ runs over all 
typed $k$-strict partitions.
\end{prop}
\begin{proof}
It follows from \cite[Thm.\ 3.2]{BKT1} that the elements $b_\la$ and
the single eta polynomials $\Eta_\la(c)$ for $\la$ a typed $k$-strict
partition form two $\Z$-bases of $B^{(k)}$. We deduce that these two
families are also $\Z[t]$-bases of $B^{(k)}[t]$. By expanding the
raising operator definition of $\Eta_\la(c\, |\, t)$, we obtain that
\[
\Eta_\la(c\, |\, t) = b_\la + \sum_{\mu} a_{\la\mu}\, b_\mu
\]
where $a_{\la\mu}\in \Z[t]$ and the sum is over typed $k$-strict
partitions $\mu$ with $\mu \succ \la$ in dominance order or
$|\mu|<|\la|$.  Therefore, the $\Eta_\la(c\, |\,t)$ for $\la$ typed
and $k$-strict form another $\Z[t]$-basis of $B^{(k)}[t]$.
\end{proof}

\subsection{The left weak Bruhat order on $\wt{W}_\infty$}
\label{lwbo}

The length of an element $w$ in $\wt{W}_\infty$ is
denoted by $\ell(w)$. It follows, for example, from \cite[p.\ 253]{BB} that 
\[
\ell(w)= \#\{i<j\ |\ w_i>w_j\} + \sum_{i\, :\, w_i<0}(|w_i|-1)
\]
for each $w\in \wt{W}_\infty$. We deduce the following lemma.

\begin{lemma}
\label{WDlem}
Suppose that $w$ is a $k$-Grassmannian element of $\wt{W}_\infty$. 

\medskip
\noin
{\em (a)} We have $\ell(s_0 w)<\ell(w)$ if and only if
 $w=(\cdots \ov{2} \cdots)$.

\medskip
\noin
{\em (b)} Assume that $i\geq 1$. We have $\ell(s_iw)<\ell(w)$ 
if and only if $w$ has one of the following four forms:
\[
(\cdots i+1 \cdots i \cdots)  \ \ \
(\cdots i \cdots \ov{i+1} \cdots) \ \ \
(\ov{i} \cdots \ov{i+1} \cdots)  \ \ \
(\cdots \ov{i+1} \cdots i \cdots).  
\]
\end{lemma}

Let $\la$ and $\mu$ be two typed $k$-strict partitions, set
$w:=w_\la$, $w':=w_\mu$, $\beta:=\beta(\la)$, $\beta':=\beta(\mu)$,
and assume that $w=s_iw'$ holds for some simple reflection $s_i\in
\wt{W}_\infty$. It follows that $\mu\subset\la$, so $\mu$ is obtained
by removing a single box from $\la$, and hence $\mu_p=\la_p-1$ for
some $p\geq 1$ and $\mu_j=\la_j$ for all $j\neq p$. Moreover, we must
have $\type(\la)+\type(\mu)\neq 3$.

Using Lemma \ref{WDlem}, we distinguish seven possible cases for $w$,
discussed below. In each case, the properties listed follow
immediately from equations (\ref{laweq}), (\ref{Cweq}), and
({\ref{css2w}). First, we consider the four cases with $i\geq 1$:

\medskip
\noin
(a) $w = (\cdots i+1 \cdots i \cdots)$.  In this case 
$\cC(\la)=\cC(\mu)$, $\be_p=i$, $\be'_p=i+1$, while 
$\be_j=\be_j'$ for all $j\neq p$.

\medskip
\noin 
(b) $w = (\cdots i \cdots \ov{i+1} \cdots)$.  In this case
$\cC(\la)=\cC(\mu)$, $\be_p = -i$, $\be'_p=-i+1$, and
$\be_j=\be_j'$ for all $j\neq p$.

\medskip
\noin 
(c) $w = (\ov{i} \cdots \ov{i+1} \cdots)$.  In this case
$w_1=\ov{i}$, $\type(\la)=2$ if $i \geq 2$, $\cC(\la)=\cC(\mu)$, $\be_p
=-i$, $\be'_p = -i+1$, and $\be_j=\be_j'$ for all $j\neq p$.

\medskip
\noin (d) $w = (\cdots \ov{i+1} \cdots i \cdots)$. We distinguish
two subcases here: Case (d1): $w_1\neq\ov{i+1}$. Then
$\cC(\la)=\cC(\mu)\cup\{(p,q)\}$, where $w_{k+p}=\ov{i+1}$ and
$w_{k+q}=i$. It follows that $\be_p=-i$, $\be_q=i$, $\be'_p=-i+1 =
\be_p+1$, and $\be'_q=i+1 = \be_q+1$, while $\be_j=\be'_j$ for all $j
\notin \{p,q\}$. Case (d2): $w_1=\ov{i+1}$ and we have
$w^{-1}(i)>k$. In this case $\type(\la)=2$, $\cC(\la)=\cC(\mu)$,
$\be_p =i$, $\be'_p=i+1$, and $\be_j=\be_j'$ for all $j\neq p$.

\medskip

Next, we consider the three cases where $i=0$. 

\medskip
\noin
(e) $w = (\wh{1} \cdots \ov{2} \cdots)$.  In this case 
$\cC(\la)=\cC(\mu)$, $\be_p =-1$, and $\be'_p=0$ if $w_1=1$, while
$\be'_p=1$ if $w_1=\ov{1}$. We also have $\be_j=\be_j'$ for all $j\neq p$.

\medskip
\noin 
(f) $w = (\ov{2} \cdots \wh{1} \cdots)$.  In this case
$\cC(\la)=\cC(\mu)$, $\be'_p =2$, and $\be_p=0$ if $w_{k+p}=\ov{1}$,
while $\be_p=1$ if $w_{k+p}=1$. We also have $\be_j=\be_j'$ for all
$j\neq p$.

\medskip
\noin (g) $w = (\cdots \ov{2} \ov{1} \cdots)$, with $|w_1|>2$.  In
this case $\cC(\la)=\cC(\mu)\cup\{(p,p+1)\}$, where $w_{k+p}=\ov{2}$
and $w_{k+p+1}=\ov{1}$. It follows that $\la_p=k+1$, $\la_{p+1}=k$,
$\be_p=-1$, $\be_{p+1}=0$, while $\mu_p=\mu_{p+1}=k$, $\be_p'=1$,
$\be'_{p+1}=2$, and $\be_j=\be'_j$ for all $j \notin
\{p,p+1\}$.

\subsection{Double eta polynomials and divided differences}

We are now ready to establish the fundamental result about the 
compatibility of the polynomials $\Eta_\la(c\, |\, t)$ with left
divided differences.

\begin{prop}
\label{uniq}
Let $\la$ and $\mu$ be typed $k$-strict partitions such that 
$|\la|=|\mu|+1$ and $w_\la=s_iw_{\mu}$ for some simple
reflection $s_i\in \wt{W}_\infty$. Then we have
\[
\partial_i\Eta_\la(c\,|\, t)  = \Eta_{\mu}(c\,|\, t)
\]
in $B^{(k)}[t]$.
\end{prop}
\begin{proof}
Let $w=w_\la$ and $w'=w_\mu$, where $\la$ and $\mu$ are typed and such 
that $w=s_iw'$ holds. We are in the situation of \S \ref{lwbo}, hence
$\mu_p=\la_p-1$ for some $p\geq 1$ and $\mu_j=\la_j$ for all $j\neq
p$. Set $\be=\be(\la)$. Let $\epsilon_j$ denote the $j$-th standard
basis vector in $\Z^\ell$. We now distinguish the following cases.

\medskip
\noin
{\bf Case 1.} $\type(\la)=\type(\mu)=0$.

\medskip
Note that we have $|w_1|=|w'_1|=1$, and hence $i\geq 2$ and
$\ell_k(\la) = \ell_k(\mu)$. We must be in one among cases (a), (b),
or (d1) of \S \ref{lwbo}.  In cases (a) or (b), it follows from
Propositions \ref{prop1} and \ref{prop1hN} and the Leibnitz rule that
for any integer sequence $\al=(\al_1,\ldots,\al_\ell)$, we have
\begin{align*}
\partial_i \wh{c}^{\be(\la)}_\al &=
\wh{c}^{(\be_1,\ldots,\be_{p-1})}_{(\al_1,\ldots,\al_{p-1})}
\left(\partial_i(\wh{c}^{\be_p}_{\al_p})\wh{c}^{(\be_{p+1},\ldots,\be_\ell)}_{(\al_{p+1},\ldots,\al_\ell)}
+s_i(\wh{c}^{\be_p}_{\al_p})\partial_i
(\wh{c}^{(\be_{p+1},\ldots,\be_\ell)}_{(\al_{p+1},\ldots,\al_\ell)})\right)
\\ &=\wh{c}^{(\be_1,\ldots,\be_{p-1})}_{(\al_1,\ldots,\al_{p-1})}
\left(\wh{c}^{\be_p+1}_{\al_p-1}\wh{c}^{(\be_{p+1},\ldots,\be_\ell)}_{(\al_{p+1},\ldots,\al_\ell)}
+s_i(\wh{c}^{\be_p}_{\al_p})\cdot 0\right)
=\wh{c}^{(\be_1,\ldots,\be_p+1,\ldots,\be_\ell)}_{(\al_1,\ldots,\al_p-1,\ldots,\al_\ell)}
= \wh{c}^{\be(\mu)}_{\al-\epsilon_p}.
\end{align*}
Since $\la-\epsilon_p=\mu$,
it follows that if $R$ is any raising operator, then
\[
\partial_i (R \star\wh{c}^{\be(\la)}_\la) = \partial_i (\ov{c}^{\be(\la)}_{R\la}) = 
\ov{c}^{\be(\mu)}_{R\la-\epsilon_p} = R \star\wh{c}^{\be(\mu)}_\mu.
\]
As $R^\la= R^\mu$, we deduce that 
\[
\partial_i \Eta_\la(c\,|\, t) = 2^{-\ell_k(\la)}\partial_i (R^\la \star
\wh{c}^{\be(\la)}_\la) = 2^{-\ell_k(\mu)} R^\mu \star
\wh{c}^{\be(\mu)}_\mu = \Eta_\mu(c\,|\, t).
\]
In case (d1), for any integer sequence $\al=(\al_1,\ldots,\al_\ell)$,
we compute that
\begin{align*}
\partial_i \wh{c}^{\be(\la)}_\al &= \partial_i
\wh{c}^{(\be_1,\ldots,-i,\ldots,i,\ldots,\be_{\ell})}_{(\al_1,\ldots,\al_p,\ldots,\al_q,
  \ldots,\al_\ell)} \\ &=
\wh{c}^{(\be_1,\ldots,-i+1,\ldots,i+1,\ldots,\be_{\ell})}_{(\al_1,\ldots,\al_p-1,\ldots,\al_q,\ldots,\al_\ell)}
+
\wh{c}^{(\be_1,\ldots,-i+1,\ldots,i+1,\ldots,\be_{\ell})}_{(\al_1,\ldots,\al_p,\ldots,\al_q-1,\ldots,\al_\ell)}
= \wh{c}^{\be(\mu)}_{\al-\epsilon_p} +
\wh{c}^{\be(\mu)}_{\al-\epsilon_q}.
\end{align*}
This follows from the Leibnitz rule, as in the proof of Proposition
\ref{prop1}(b). If $R$ is any raising operator, then since $i\geq 2$
we must have $q>\ell_k(\la)$ and hence $q\notin\supp_m(RR_{pq})$,
where $m=\ell_k(\mu)+1$. As $\la-\epsilon_p=\mu$, we deduce that
\[
\partial_i(R \star \wh{c}^{\be(\la)}_\la) = 
\partial_i (\ov{c}^{\be(\la)}_{R\la}) = 
\ov{c}^{\be(\mu)}_{R\la-\epsilon_p} + \ov{c}^{\be(\mu)}_{R\la-\epsilon_q}
=  R \star \wh{c}^{\be(\mu)}_\mu + RR_{pq} \star \wh{c}^{\be(\mu)}_\mu.
\]
Since $R^\la+R^\la R_{pq} = R^\mu$, it follows that
$\partial_i \Eta_\la(c\,|\, t)  =  \Eta_\mu(c\,|\, t)$.

\medskip
\noin
{\bf Case 2.} $\type(\la)=0$ and $\type(\mu)>0$.

\medskip
In this case, we have $|w_1|=1$ and $|w'_1|>1$, so $i\in \{0, 1\}$.  We
must be in one of cases (b), (c), or (e) of \S \ref{lwbo}, hence
$\cC(\la)=\cC(\mu)$.  We also have $\la_p=k+1$ and $\la_{p+1}<k$, so
$(p,p+1)\notin \cC(\la)$, $\be_p(\la)=-1$, $\be_p(\mu)\in\{0,1\}$, and
$\ell_k(\la) = \ell_k(\mu)+1$.

Observe that, for any integer sequence
$\al=(\al_1,\ldots,\al_\ell)$, we have
\begin{align*}
\partial_i \wh{c}^{\be(\la)}_\al &=
\wh{c}^{(\be_1,\ldots,\be_{p-1})}_{(\al_1,\ldots,\al_{p-1})}
\left(\partial_i(\wh{c}^{-1}_{\al_p})
\wh{c}^{(\be_{p+1},\ldots,\be_\ell)}_{(\al_{p+1},\ldots,\al_\ell)}
+s_i(\wh{c}^{-1}_{\al_p})\partial_i
(\wh{c}^{(\be_{p+1},\ldots,\be_\ell)}_{(\al_{p+1},\ldots,\al_\ell)})\right)
\\ &=
\wh{c}^{(\be_1,\ldots,\be_{p-1})}_{(\al_1,\ldots,\al_{p-1})}\partial_i(\wh{c}^{-1}_{\al_p})
\wh{c}^{(\be_{p+1},\ldots,\be_\ell)}_{(\al_{p+1},\ldots,\al_\ell)}.
\end{align*}
We now compute using Propositions \ref{prop1} and \ref{prop1hN}(a) that
\[
\partial_1\wh{c}_q^{-1}= 
\begin{cases}
2a_{q-1}^0 & \text{if $q\neq k+1$} \\
2f_k & \text{if $q=k+1$}.
\end{cases}
\]
Proposition \ref{prop1}(a) and equation (\ref{cmin}) give
\[
\partial_0\wh{c}_q^{-1}=  
\begin{cases}
2a_{q-1}^1 & \text{if $q\neq k+1$} \\
2\wt{f}^1_k & \text{if $q=k+1$}.
\end{cases}
\]
The rest is straightforward from the definitions, arguing as in Case 1.

\medskip
\noin
{\bf Case 3.} $\type(\la)>0$ and $\type(\mu)=0$. 

\medskip
We have $|w_1|>1$ and $|w'_1|=1$, so $i\in\{0,1\}$, and we are in one of
cases (a), (d2), or (f) of \S \ref{lwbo}, hence $\cC(\la)=\cC(\mu)$.
We also have $\la_p=k$, $\be_p(\la)\in \{0,1\}$, $\be_p(\mu)=2$, and
$\ell_k(\la) = \ell_k(\mu)$. Recall that $\wh{c}_p^r = c_p^r$ whenever
$p\leq k$, $b_k^1=c_k^1-\wt{b}_k$, $\wt{b}_k^1=c_k^1-b_k$,
and $a^s_p=c^s_p-\frac{1}{2}c_p$.  We deduce the calculations
\begin{gather*}
\partial_0 b_k = \partial_0 \wt{b}_k = \partial_1 b^1_k = \partial_1
\wt{b}_k^1 = c_{k-1}^2 \\ \partial_0 a^0_p = \partial_1 a_p^1 =
c_{p-1}^2.
\end{gather*}
As in the previous cases, it follows that 
$\partial_i \Eta_\la(c\,|\, t) = \Eta_\mu(c\,|\, t)$.

\medskip
\noin
{\bf Case 4.} $\type(\la)=\type(\mu)>0$.

\medskip
We have $|w_1|>1$ and $|w'_1|>1$. If $i\geq 2$, then we must be in one
of cases (a), (b), (c), or (d1) of \S \ref{lwbo}, and the result is
proved by arguing as in Case 1. It remains to study (i) case (d1) with
$w=(\cdots \ov{2}1 \cdots)$ and $i=1$, or (ii) case (g) with $w =
(\cdots \ov{2} \ov{1} \cdots)$ and $i=0$. In both of these subcases,
we have $\cC(\la)=\cC(\mu)\cup\{(p,p+1)\}$, $\ell_k(\la) =
\ell_k(\mu)+1$, $\la_p=k+1$, $\la_{p+1}=\mu_p=\mu_{p+1}=k$,
$\be_p(\la)=-1$, $\be_{p+1}(\mu)=2$, and $\be_j(\la)=\be_j(\mu)$ for
all $j \notin \{p,p+1\}$. In subcase (i), we have $\be_{p+1}(\la)=1$
and $\be_p(\mu)=0$, while in subcase (ii), we have $\be_{p+1}(\la)=0$
and $\be_p(\mu)=1$.

To deal with subcase (i), we argue as in Case 1 (d1), this time
applying the identities (\ref{12equ})--(\ref{12Bequ}). There is now an
added complication: we must show that the total contribution from the
four residual terms that appear with a negative sign in equations
(\ref{12equ})--(\ref{12Bequ}) vanishes. To prove this, we may assume
that $\la$ has length $p+1$, and consider the effect of the raising
operators $R$ in the expansion of $R^\la$ which involve only basic
operators $R_{ij}$ with $i=p$ or $j=p+1$.

An integer sequence $\al=(\al_1,\ldots,\al_\ell)$ is a {\em
  composition} if $\al_i\geq 0$ for all $i$.  For any composition
$\alpha$, let $|\al|:=\sum_i\al_i$ and $\#\al$ denote the number of
non-zero components $\al_i$. The relevant raising operator expression
is
\begin{align*}
\Psi &:=\left(\prod_{i=1}^{p-1}\frac{1-R_{ip}}{1+R_{ip}}\right)\left(
\prod_{i=1}^{p-1}\frac{1-R_{i,p+1}}{1+R_{i,p+1}}\right)
\frac{1-R_{p,p+1}}{1+R_{p,p+1}} \\
&= \sum_{\al',\al,d\geq 0}(-1)^{|\al'|+|\al|+d}\,2^{\#(\al',\al,d)}\left(
\prod_{i=1}^{p-1}R_{ip}^{\al'_i}\,R_{i,p+1}^{\al_i}\right)R_{p.p+1}^d
\end{align*}
where the sum is over all compositions
$\al'=(\al'_1,\ldots,\al'_{p-1})$,
$\al=(\al_1,\ldots,\al_{p-1})$, and integers $d\geq 0$.  If
$\nu=(\nu_1,\ldots,\nu_{p-1})$ is a fixed integer vector, then
$(-1/2)$ times the total residual term in the expansion of $\partial_1(\Psi
\, \ov{c}^{(\be_1,\ldots,\be_{p+1})}_{(\nu,k+1,k)})$ is equal to
\begin{equation}
\label{Seq}
S_\nu := \sum_{\al',\al,d\geq 0}(-1)^{|\al'|+|\al|+d}\,2^{\#(\al',\al,d)}\,
\ov{c}_{\nu+\al'+\al}^{(\be_1,\ldots,\be_{p-1})}\,\ov{a}_{k-|\al'|+d}\,\ov{a}_{k-|\al|-d}.
\end{equation}
The two factors $\ov{a}_q$ in each summand of (\ref{Seq}) are equal to
$a_q$, $f_k$, or $\wt{f}_k$, according to the equations
(\ref{12equ})--(\ref{12Bequ}) and depending on the choice of $\al'$,
$\al$, and $d$, as in Definition \ref{Etadef}. We now make a 
change of variables in the sum (\ref{Seq}) by setting $\rho:=\al'+\al$
and $r:=|\al|+d$, to obtain
\begin{equation}
\label{Seq2}
S_\nu = \sum_{\rho\geq 0}(-1)^{|\rho|}\,\ov{c}^{(\be_1,\ldots,\be_{p-1})}_{\nu+\rho}\sum_{r=0}^k
T(\rho,r)
\end{equation}
where the first sum is over all compositions $\rho$ and
\[
T(\rho,r):= (-1)^r\,
\ov{a}_{k-|\rho|+r}\,\ov{a}_{k-r}\sum_{{0\leq \al\leq \rho}\atop {|\al|\leq r}} 
(-1)^{|\al|}\,2^{\#(\rho-\al,\al,r-|\al|)}.
\]

We compute that
\[
\sum_{r=0}^k T(0,r) =
f_k\wt{f}_k + 2\sum_{r=1}^k (-1)^r a_{k+r} a_{k-r} = 
b_k\wt{b}_k+\sum_{r=1}^k(-1)^r b_{k+r}b_{k-r} \in J^{(k)}.
\]
It follows that the sum of the terms in (\ref{Seq2}) with $\rho=0$
vanishes in $B^{(k)}[t]$. We claim that the sum of all the remaining terms
in (\ref{Seq2}) is identically zero.

\begin{lemma}
\label{keylemma}
Let $\rho$ be a non-zero composition. If $r>|\rho|$, then
$T(\rho,r)=0$, while if $0 \leq r \leq |\rho|$, then $T(\rho,r) +
T(\rho,|\rho|-r)=0$.
\end{lemma}
\begin{proof}
The argument is based on the elementary identity
\begin{equation}
\label{elem}
\sum_{i=0}^s (-1)^i\, 2^{\#(s-i,i)} = \delta_{s,0}.
\end{equation}
By multiplying together a finite number of equations of the form
(\ref{elem}), we obtain
\begin{equation}
\label{elem2}
\sum_{0\leq \al\leq \rho}(-1)^{|\al|}
\,2^{\#(\rho-\al,\al)} = \delta_{\rho,0}
\end{equation}
for any composition $\rho$, where the sum is over all compositions
$\al$ with $\al\leq \rho$.

Assume now that $\rho\neq 0$. If $r> |\rho|$, then using (\ref{elem2}) gives
\[
T(\rho,r) = (-1)^r \,
a_{k-|\rho|+r}\,a_{k-r}\cdot 2  \sum_{0\leq \al\leq \rho}(-1)^{|\al|}
\,2^{\#(\rho-\al,\al)}=0.
\]
If $0<r<|\rho|$, then 
\begin{equation}
\label{T1}
T(\rho,r) = (-1)^r a_{k-|\rho|+r}\, a_{k-r}  \sum_{{0\leq \al\leq \rho}\atop {|\al|\leq r}} 
(-1)^{|\al|}\,2^{\#(\rho-\al,\al,r-|\al|)}
\end{equation}
and the substitution $\al':=\rho - \al$ gives
\begin{equation}
\label{T2}
T(\rho,|\rho|-r) = (-1)^r a_{k-r} \, a_{k-|\rho|+r} 
\sum_{{0\leq \al'\leq \rho}\atop {|\al'|\geq r}} 
(-1)^{|\al'|}\,2^{\#(\rho-\al',\al',|\al'|-r)}.
\end{equation}
Adding (\ref{T1}) to (\ref{T2}) and applying (\ref{elem2}) gives
$T(\rho,r) + T(\rho,|\rho|-r) = 0$.

Finally, we have 
\[
T(\rho,0) = 2^{\#\rho}\,a_{k-|\rho|}\wt{f}_k
\]
while
\begin{align*}
T(\rho,|\rho|) &= (-1)^{|\rho|} \ov{a}_k \ov{a}_{k-|\rho|} 
\sum_{0\leq \al\leq \rho} 
(-1)^{|\al|}\,2^{\#(\rho-\al,\al,|\rho|-|\al|)} \\
&= 2^{\#\rho}\, f_ka_{k-|\rho|} + 
2 a_k a_{k-|\rho|} 
\sum_{{0\leq \al\leq \rho}\atop {\al\neq \rho}} 
(-1)^{|\rho|-|\al|}\,2^{\#(\rho-\al,\al)}.
\end{align*}
Since $f_k+\wt{f}_k=2a_k$, adding the previous equations and
applying (\ref{elem2}) again shows that $T(\rho,0) + T(\rho,|\rho|) =
0$.
\end{proof}

Using Lemma \ref{keylemma} in equation (\ref{Seq2}) proves the claim,
and completes the argument in subcase (i). The proof for subcase (ii)
is similar, this time using the equations
(\ref{12equ2})--(\ref{12Bequ2}) and the relation
\[
\wt{f}^1_kf^1_k + 2\sum_{r=1}^k (-1)^r a^1_{k+r} a^1_{k-r} = 0
\]
in $B^{(k)}[t]$, which is easily checked.
\end{proof}

\begin{remark}
\label{rm1}
The proof of Proposition \ref{uniq} establishes that the equality
$\partial_i\Eta_\la(c\,|\, t) = \Eta_{\mu}(c\,|\, t)$ holds in
$\Z[b,t]$ in all cases of \S \ref{lwbo} except case (d1) with
$i=1$ or case (g) with $i=0$. In each of the latter two cases,
we need to use the relation (\ref{relation2}) exactly once. The
basic example that illustrates this is the equality
\begin{equation}
\label{basic}
\partial_i\Eta_{(k+1,k)}(c\,|\, t)  = \Eta_{(k,k)}(c\,|\, t),
\end{equation}
where $i\in\{0,1\}$ and both of the indexing partitions have the same
(positive) type. Equation (\ref{basic}) is true in $B^{(k)}[t]$, but
fails in $\Z[b,t]$.
\end{remark}

\subsection{The polynomials $\wh{H}_\la(c\, |\, t)$}
\label{Hhat}

In this subsection we define and study a closely related family of 
polynomials $\wh{\Eta}_\la(c \, |\, t)$ indexed by 
$k$-strict partitions $\la$. The polynomials $\wh{\Eta}_\la(c)
:=\wh{\Eta}_\la(c\, |\, 0)$ were studied in \cite[\S 5.2]{BKT2}.
As explained in op.\ cit., $\wh{\Eta}_\la(c)$ represents the 
cohomology class of a certain Zariski closed subset $Y_\la$ of
$\OG(n-k,2n)$, which is either a Schubert variety or a union of two 
Schubert varieties. The double polynomials $\wh{\Eta}_\la(c \, |\, t)$
similarly represent the $T_n$-equivariant cohomology class 
$[Y_\la]^{T_n}$ in $\HH^*_{T_n}(\OG)$, under the geometrization map
$\pi_n$ defined in \S \ref{geomapD}; this follows immediately
from their definition below and Theorem \ref{mainthm}.

If $\la$ is any $k$-strict partition, define the finite set 
of pairs
\[
\cC(\la) :=\{ (i,j)\in \N\times\N \ |\ 1\leq i<j \ \ \text{and} \ \ 
\la_i+\la_j \geq 2k+j-i\}
\]
and the sequence $\ov{\beta}(\la)=\{\ov{\beta}_j(\la)\}_{j\geq 1}$ by
\[
\ov{\be}_j(\la):=k-\la_j+\#\{i<j\ |\ (i,j)\notin \cC(\la)\} + 
\begin{cases} 1 & \text{if $\la_j\leq k$} \\
0 & \text{if $\la_j > k$}, 
\end{cases} \ \ \text{for all} \ \ j\geq 1.
\]
We have $\cC(\la)=\cC(\ov{\la})$ and
$\ov{\beta}(\la)=\beta(\ov{\la})$, where $\ov{\la}$ denotes the unique
typed $k$-strict partition which has the same shape as $\la$, and with
the property that $\be_j(\ov{\la})\neq 0$, for each $j\geq 1$. For 
comparison with \cite{BKT2}, we note that 
$\ov{\be}_j(\la) = \ov{p}_j(\la)-n$, where $\ov{p}_j(\la)$ is the
function defined in the introduction of op.\ cit.

If $\la_i=k$ for some index $i$, then we agree that $\Eta_\la(c\, |\,
t)$ and $\Eta'_\la(c\, |\, t)$ denote the double eta polynomials
indexed by $\la$ of type 1 and 2, respectively; otherwise,
$\Eta_\la(c\, |\, t)$ denotes the associated double eta polynomial
indexed by $\la$ of type zero. We define the raising operator
expression $R^\la$ by equation (\ref{Req}), as before.

\begin{defn}
\label{hatEtadef}
For any $k$-strict partition $\la$, let $m:=\ell_k(\la)+1$ and
$\ov{\be}:=\ov{\be}(\la)$. If $R$ is any raising operator appearing in
the expansion of the power series $R^\la$ and $\nu:=R\la$, define
\[
R \star \wh{c}^{\ov{\be}(\la)}_{\la} := 
\ov{c}_{\nu_1}^{\ov{\be}_1}\cdots\ov{c}^{\ov{\be}_\ell}_{\nu_\ell}
\]
where for each $i\geq 1$, 
\[
\ov{c}_{\nu_i}^{\ov{\be}_i}:= 
\begin{cases}
c_{\nu_i}^{\ov{\be}_i} & \text{if $i\in\supp_m(R)$}, \\
\wh{c}_{\nu_i}^{\ov{\be}_i} & \text{otherwise}.
\end{cases}
\]
The polynomial $\wh{\Eta}_\la(c \, |\, t)$ is defined by 
\[
\wh{\Eta}_\la(c \, |\, t) := 2^{-\ell_k(\la)}R^\la\star\wh{c}^{\ov{\be}(\la)}_{\la}
= \begin{cases}
\Eta_\la(c\, |\, t) + \Eta'_\la(c\, |\, t) 
& \text{if $\la_i=k$ for some $i$}, \\
\ \quad \Eta_\la(c\,|\, t) & \text{otherwise}.
\end{cases}
\]
\end{defn}

Table \ref{table2} lists the double eta hat polynomials associated 
to the Grassmannian elements in $\wt{W}_3$. We have retained the 
negative powers of $2$ in this table for clarity.

{\small{
\begin{table}[t]
\caption{Double eta hat polynomials for Grassmannian $w\in \wt{W}_3$}
\centering
\begin{tabular}{|c|c|c|c|} \hline
$w$ & $\la$ & $\ov{\beta}$ & $\wh{H}_\la(c\, |\, t)$
\\ \hline

$123$ &   &  & $1$ \\

$213$,  $\ov{2}\ov{1}3$ & $1$ & $(1,3)$  & $c_1+h_1^1$ \\
 
$\ov{1}\ov{2}3$ & 2 & $(-1,3)$ &  $\frac{1}{2}(c_2+2\wt{b}_1e_1^1)$ \\

$312$, $\ov{3}\ov{1}2$  & $(1,1)$ & $(1,2)$ & 
$(c_1+h_1^1)(c_1+h_1^2) - (c_2+c_1h_1^1+h_2^1)$ \\

$\ov{1}\ov{3}2$ & 3 & $(-2,2)$ & $\frac{1}{2}(c_3+c_2e_1^2+2\wt{b}_1e_2^2)$ \\

$3\ov{2}\ov{1}$,$\ov{3}\ov{2}1$  &  $(2,1)$  &  $(-1,1)$ 
& $\frac{1}{2}((c_2+2\wt{b}_1e_1^1)(c_1+h^1_1)- 2(c_3+c_2e_1^1))$ \\

$2\ov{3}\ov{1}$, $\ov{2}\ov{3}1$  &  $(3,1)$  &  $(-2,1)$  &  
$\frac{1}{2}((c_3+c_2e_1^2+2\wt{b}_1e_2^2)(c_1+h_1^1)-2(c_4+c_3e_1^2+c_2e_2^2))$ \\

$1\ov{3}\ov{2}$ & $(3,2)$ &  $(-2,-1)$ & 
$\frac{1}{4}((c_3+c_2e_1^2+2\wt{b}_1e_2^2)(c_2+2b_1e_1^1)$ \\

&&& $-2(c_4+c_3e_1^2+c_2e_2^2)(c_1+e_1^1)+2(c_5+c_4e_1^2+c_3e_2^2))$ \\

\hline

$132$  &  $1$  &  $2$  &  $c_1+h_1^2$ \\

$231$, $\ov{2}3\ov{1}$ &  $2$  &  $1$  &  $c_2+c_1h_1^1+h_2^1$ \\

$\ov{1}3\ov{2}$  &  $3$  &  $-1$  &  $\frac{1}{2}(c_3+2\wt{b}_2e_1^1)$ \\

$\ov{1}2\ov{3}$  &  $4$  &  $-2$  
&  $\frac{1}{2}(c_4+c_3e_1^2+2\wt{b}_2e_2^2)$ \\

\hline
\end{tabular}
\label{table2}
\end{table}}}

\begin{prop}
\label{uniq0}
Let $\la$ and $\mu$ be two $k$-strict partitions with $|\la|=|\mu|+1$.
Assume that there exist a simple reflection $s_i\in \wt{W}_\infty$
and a choice of type assigned to $\la$ and $\mu$ such that
$w_\la=s_iw_{\mu}$ in $\wt{W}_\infty$. Then the following 
assertions hold in $\Z[b,t]$.

\medskip
{\em (i)}
If $\type(\la)=\type(\mu)=0$, then $i\geq 2$ and 
\[
\partial_i\wh{\Eta}_\la(c\,|\, t)  = \wh{\Eta}_{\mu}(c\,|\, t).
\]

{\em (ii)}
If $\type(\la)=0$ and $\type(\mu)>0$, then $i\in \{0, 1\}$ and 
\[
(\partial_0+\partial_1)\wh{\Eta}_\la(c\,|\, t)  = \wh{\Eta}_{\mu}(c\,|\, t).
\]

{\em (iii)}
If $\type(\la)>0$ and $\type(\mu)=0$, then $i\in \{0, 1\}$ and
\[
\partial_0\wh{\Eta}_\la(c\,|\, t)  =
\partial_1\wh{\Eta}_\la(c\,|\, t)  = \wh{\Eta}_{\mu}(c\,|\, t).
\]

{\em (iv)}
If $\type(\la)=\type(\mu)>0$, then  
\[
\partial_i\wh{\Eta}_\la(c\,|\, t)  = \wh{\Eta}_{\mu}(c\,|\, t),
\]

if $i\geq 2$, and 
\[
(\partial_0 + \partial_1)\wh{\Eta}_\la(c\,|\, t)  = \wh{\Eta}_{\mu}(c\,|\, t),
\]

if $i\in \{0, 1\}$. 
\end{prop}
\begin{proof}
We will only give the outline of the proof of claims (i)--(iv) here, as
the argument is very similar to the proof of Proposition \ref{uniq},
only easier, because the relation (\ref{relation2}) is never used.  
Recall the seven possible cases (a)--(g) for $w_\la$ from \S
\ref{lwbo}.

For claim (i), or claim (iv) when $i\geq 2$, we must be in one among 
cases (a), (b), (c), or (d1) of \S \ref{lwbo}, and the proof is exactly
as in Proposition \ref{uniq}. We are left with  examining the claims (ii), 
(iii), and (iv) when $i\in \{0,1\}$. For claim (ii), we must be in one of 
cases (b), (c), or (e), and we use equations (\ref{beqN}) and (\ref{beqhN}).
For claim (iii), we are in one of cases (a), (d2), or (f), and use
the computation
\[
\partial_0 c_p^1 = \partial_1  c_p^1 = c_{p-1}^2.
\]
Finally, for claim (iv) we must be in case (d1) with 
$w=(\cdots \ov{2}1 \cdots)$ and $i=1$, or in case (g) with 
$w = (\cdots \ov{2} \ov{1} \cdots)$ and $i=0$. The result 
follows as in claim (i), case (d1), but now using 
equations (\ref{12eqB}) and (\ref{12eqBNh}).
\end{proof}

\section{The proof of Theorem \ref{mainthm}}
\label{gddsD}

\subsection{The geometrization map}
\label{geomapD}
Let  
\[
0 \to E'\to E \to E''\to 0
\]
denote the universal exact sequence of vector bundles over
$\OG(n-k,2n)$, with $E$ the trivial bundle of rank $2n$ and $E'$ the
tautological subbundle of rank $n-k$. For $0\leq j \leq 2n$, define
the subbundles $F_j$ of $E$ as in the introduction. Let
$OM_n:=ET_n\times^{T_n}\OG$ denote the Borel mixing space for the
action of the torus $T_n$ on $\OG$. The $T_n$-equivariant vector
bundles $E',E,E'',F_j$ over $\OG$ induce vector bundles over $OM_n$,
and their equivariant Chern classes in
$\HH^*(OM_n,\Z)=\HH^*_{T_n}(\OG(n-k,2n))$ are denoted by $c_p^T(E')$,
$c^T_p(E)$, $c^T_p(E'')$, and $c^T_p(F_j)$.

The class $c^T_p(E-E'-F_j)$ for $p\geq 0$ is defined by the total
Chern class equation
$$c^T(E-E'-F_j):=c^T(E)c^T(E')^{-1}c^T(F_j)^{-1}.$$ Let
$\tm_i:=-c^T_1(F_{n+1-i}/F_{n-i})$ for $1\leq i \leq n$.  Following
\cite[\S 10]{IMN1} and \cite[\S 7]{T2}, we define the {\em
  geometrization map} $\pi_n$ as the $\Z[t]$-algebra homomorphism
\[
\pi_n : B^{(k)}[t] \to \HH^*_{T_n}(\OG(n-k,2n))
\]
determined by setting
\begin{gather*}
\pi_n(b_p):= \begin{cases}
c^T_p(E-E'-F_n)  & \text{if $p<k$}, \\
\frac{1}{2}c^T_p(E-E'-F_n)  & \text{if $p>k$},
\end{cases} \\
\pi_n(b_k):=\frac{1}{2}(c^T_k(E-E'-F_n)+c^T_k(E_n-E')), \\
\pi_n(\wt{b}_k):=\frac{1}{2}(c^T_k(E-E'-F_n)-c^T_k(E_n-E')), \\
\pi_n(t_i):= \begin{cases}
\tm_i & \text{if $1\leq i\leq n$}, \\
0 & \text{if $i>n$}. \end{cases}
\end{gather*}
Here $E_n$ denotes a maximal isotropic subbundle of the (pullback of)
$E$ to the complete flag variety, which is in the same family as
$F_n$. Note that the images of the elements $b_p$, $\wt{b}_k$, $c_p$
in the ring of type D Billey-Haiman Schubert polynomials \cite{BH} are
given in \cite[\S 5]{BKT2} by the power series $\eta_p(x\,;\, y)$,
$\eta'_k(x\,;\, y)$, $\ti_p(x\,;\, y)$, respectively, and, using
this, the equations defining $\pi_n$ are derived in \cite[\S
7.4]{T2}. For more information on the image of the double eta
polynomials $\Eta_\la(c\, |\, t)$ in the ring of type D double Schubert
polynomials of \cite{IMN1}, see \cite[\S 4.5]{T3}.

The above equations imply that $\pi_n(c_p) = c^T_p(E-E'-F_n)$
for all $p\geq 0$. Since $\tm_1\ldots,\tm_r$ are the (equivariant)
Chern roots of $F_{n+r}/F_n$ for $1\leq r \leq n$, it follows that
\begin{equation}
\label{geomeq}
\pi_n(c^r_p) = \sum_{j=0}^pc^T_{p-j}(E-E'-F_n)h^r_j(-\tm)
= c^T_p(E-E'-F_{n+r})
\end{equation}
for $-n \leq r \leq n$. Equation (\ref{geomeq}) can be extended to any
$r\in \Z$ if we set $F_j=F_{2n}=E$ for $j>2n$ and $F_j=0$ for
$j<0$. Moreover, for $s:=p-k>0$, we have
\[
\pi_n(\wh{c}^{-s}_p) = \pi_n(c_p^{-s}+(2f_k-c_k)e_s^s(-t)) = 
c^T_p(E-E'-F_{n-s})\pm e^T(E',F_{n-s}), 
\]
where the sign depends on the choice of $f_k\in\{b_k,\wt{b}_k\}$, as
above, and the equivariant Euler class $e^T(E',F_{n-s})$ is given by
\[
e^T(E',F_{n-s}) := c^T_p(E_n/E'+F_n/F_{n-s}) =
c^T_k(E_n-E')c^T_s(F_n-F_{n-s}).
\]

The embedding of $\wt{W}_n$ into $\wt{W}_{n+1}$ defined in the
introduction induces maps of equivariant cohomology rings
$\HH^*_{T_{n+1}}(\OG(n+1-k,2n+2)) \to \HH^*_{T_n}(\OG(n-k,2n))$ which
are compatible with the morphisms $\pi_n$. We therefore obtain an
induced $\Z[t]$-algebra homomorphism
\[
\pi:B^{(k)}[t]\to \IH_T(\OG_k).
\]
The above map $\pi$ is the one that appears in Theorem \ref{mainthm},
and we proceed to show that it has the properties listed there.

\subsection{Proof of Theorem \ref{mainthm}}
\label{pointclassD}

The argument is similar to the one found in \cite[\S 6.3]{TW}, but we
include the details here for completeness.  Fix a rank $n$ and
let $$\la_0:=(n+k-1,n+k-2,\ldots,2k)$$ be the typed $k$-strict
partition associated to the $k$-Grassmannian element of maximal length
in $\wt{W}_n$. Definition \ref{Etadef} gives 
\begin{equation}
\label{topeq}
\Eta_{\la_0}(c\, |\, t) = 2^{k-n}\, R^{\la_0}_{n-k} \star c^{(1-n,2-n,\ldots,-k)}_{\la_0}
\end{equation}
where 
\[
R^{\la_0}_{n-k} := \prod_{1\leq i<j \leq n-k} \frac{1-R_{ij}}{1+R_{ij}}\,.
\]
Using (\ref{topeq}) and the equations of \S \ref{geomapD}, one checks
that $\pi_n(\Eta_{\la_0}(c\, |\, t))$ agrees with a known formula of
Kazarian \cite{Ka} for the cohomology class of the degeneracy locus
which correponds to $[X_{\la_0}]^{T_n}$. Although the final result in
\cite[App.\ D]{Ka} is expressed as a Pfaffian, this is not required
for the application here. (The equivalence of the two formulas is a
consequence of some formal Pfaffian algebra from \cite{Ka, Kn}; for a
detailed discussion of this, see \cite[App.\ A]{AF}). It follows that
\begin{equation}
\label{pteqD}
\pi_n(\Eta_{\la_0}(c\, |\, t)) = [X_{\la_0}]^{T_n}.
\end{equation}

We have shown in Proposition \ref{basisthm} that the $\Eta_\la(c\, |\,t)$ 
for $\la$ a typed $k$-strict partition form a $\Z[t]$-basis of
$B^{(k)}[t]$. Let $\wt{\cP}(k,n)$ denote the set of all typed $k$-strict
partitions whose diagrams fit inside a rectangle of size
$(n-k)\times (n+k-1)$. The elements of $\wt{\cP}(k,n)$ correspond
to the $k$-Grassmannian elements of $\wt{W}_n$ under the 
bijection described in the introduction. Let 
$w_\la$ denote the element of $\wt{W}_n$ associated to $\la$ 
under this bijection.

Following \cite[\S 6.3]{TW}, for any typed $k$-strict
partition $\la\in \wt{\cP}(k,n)$, write $w_{\la}w_{\la_0}=s_{a_1}\cdots
s_{a_r}$ as a product of simple reflections $s_{a_j}$ in $\wt{W}_n$, with
$r=|\la_0|-|\la|$. Since $w_{\la_0}^2=1$, Proposition 
\ref{uniq} implies that
\begin{equation}
\label{iteratepar}
\Eta_\la(c\,|\, t) = \partial_{a_1} \circ \cdots \circ 
\partial_{a_r}(\Eta_{\la_0}(c\, |\, t))
\end{equation}
holds in $B^{(k)}[t]$. 

The left divided differences $\delta_i$ on $\HH_{T_n}^*(\OG(n-k,2n))$
from \cite[\S 2.5]{IMN1} correspond to the operators $\partial_i$ on
$B^{(k)}[t]$, and are compatible with the geometrization map
$\pi_n:B^{(k)}[t]\to \HH_{T_n}^*(\OG(n-k,2n))$. Moreover, it is known
by \cite[Prop.\ 2.3]{IMN1} that $\delta_i([X_\la]^{T_n})=
[X_{\mu}]^{T_n}$ whenever $|\la|=|\mu|+1$ and $w_\la=s_iw_{\mu}$ for
some simple reflection $s_i$.  It follows from this and equations
(\ref{pteqD}) and (\ref{iteratepar}) that
\[
\pi_n(\Eta_\la(c\, |\, t)) = [X_\la]^{T_n}. 
\]
The vanishing property for equivariant Schubert classes (see, for
example, \cite[Prop.\ 7.7]{IMN1}) now implies that $\pi_n(\Eta_\la(c\,
|\, t))=0$ whenever $\la\notin\wt{\cP}(k,n)$ (or equivalently
$w_\la\notin \wt{W}_n$). The induced map $\pi:B^{(k)}[t]\to
\IH_T(\OG_k)$ satisfies $\pi(\Eta_\la(c\, |\, t))=\tau_\la$ for all
typed $k$-strict partitions $\la$, and is a $\Z[t]$-algebra
isomorphism because the $\Eta_\la(c\, |\, t)$ and $\tau_\la$ for $\la$
$k$-strict and typed form $\Z[t]$-bases of the respective algebras.

\subsection{A splitting theorem for $\Eta_\la(c\, |\, t)$}
\label{sov}

In this subsection, following \cite[Cor.\ 2]{TW}, we apply Theorem 
\ref{mainthm} to compare the double eta polynomials $\Eta_\la(c\, |\, t)$ of
the present paper with the general degeneracy locus formulas of
\cite[\S 6]{T1}.

The symmetric group $S_n$ is the subgroup of $\wt{W}_n$ generated by
the transpositions $s_i$ for $1\leq i \leq n-1$; we let $S_\infty :=
\cup_nS_n$ be the corresponding subgroup of $\wt{W}_\infty$. For every
permutation $u\in S_\infty$, let $\AS_u(t)$ denote the type A Schubert
polynomial of Lascoux and Sch\"utzenberger \cite{LS} indexed by $u$
(our notation follows \cite[\S 5]{T2}). The $\AS_u(t)$ for $u\in
S_\infty$ form a free $\Z$-basis of the polynomial $\Z[t]$. We deduce
from Proposition \ref{basisthm} that the products
$\Eta_\mu(c)\AS_u(-t)$ where $\mu$ ranges over all typed $k$-strict
partitions and $u\in S_\infty$ form a free $\Z$-basis of
$B^{(k)}[t]$. The following result gives the unique expansion of
(the class of) the double eta polynomial $\Eta_\la(c\, |\, t)$ in
$B^{(k)}[t]$ as a $\Z$-linear combination of this product basis.

We say that a factorization $w_\la=uv$ in $\wt{W}_\infty$ is reduced
if $\ell(w_\la)=\ell(u)+\ell(v)$. In any such factorization, the right
factor $v=w_\mu$ is also $k$-Grassmannian for some typed $k$-strict
partition $\mu$.

\begin{cor}
\label{comp}
Let $\la$ be any typed $k$-strict partition. Then we have
\begin{equation}
\label{TO}
\Eta_\la(c\, |\, t)=\sum_{uw_\mu=w_\la}\Eta_\mu(c)\AS_{u^{-1}}(-t)
\end{equation}
in the ring $B^{(k)}[t]$, where the sum is over all reduced factorizations 
$uw_\mu=w_\la$ with $u\in S_\infty$.
\end{cor}
\begin{proof}
As a special case of the splitting and degeneracy locus formulas of
\cite[\S 6]{T1}, we deduce that the polynomial on right hand side of
(\ref{TO}) represents the stable equivariant Schubert class $\tau_\la$
in $\IH_T(\OG_k)$ under the geometrization map $\pi$. The result is
therefore a direct consequence of Theorem \ref{mainthm}.
\end{proof}

It is tempting to view Corollary \ref{comp} as a separation of the
variables $b$ and $t$ in $\Eta_\la(c\, |\, t)$. However equation
(\ref{TO}) does not hold in the polynomial ring $\Z[b,t]$ for a
general $\la$, as it depends on the relations (\ref{relation1}) and
(\ref{relation2}) among the $b_p$.

\subsection{The Grassmannian $\OG(n,2n)$}
We conclude this paper with a short discussion of the situation when
$k=0$, so that $\OG=\OG(n,2n)$ parametrizes one connected component of the space
of all isotropic subspaces of $\C^{2n}$ of maximal dimension $n$. One
knows that this variety is isomorphic (in fact, projectively
equivalent) to the odd orthogonal Grassmannian $\OG(n-1,
2n-1)$. Moreover, one can arrange that this isomorphism is
torus-equivariant, and hence induces an isomorphism of equivariant
cohomology rings (see e.g.\ \cite[\S 3.5]{IMN2}). It follows that the
double theta polynomials $\Ti_\la(c\, |\, t)$ of \cite{TW} times the
appropriate negative power of $2$, which represent the equivariant
Schubert classes on $\OG(n-1, 2n-1)$, also serve as equivariant
Giambelli polynomials for $\OG(n,2n)$ (compare with \cite{IMN2}).

\end{document}